\documentclass[noinfoline]{imsart}

\RequirePackage[OT1]{fontenc}
\RequirePackage{amsthm,amsmath}
\RequirePackage[numbers]{natbib}
\RequirePackage[colorlinks,citecolor=blue,urlcolor=blue]{hyperref}

\usepackage[a4paper]{geometry}

\usepackage{a4wide} 

\usepackage{amsfonts, amssymb}
\usepackage{graphicx, color, latexsym}
\usepackage{mathtools}

\usepackage{esint}
\usepackage{dsfont}

\newtheorem{prop}{Proposition}
\newtheorem{lem}{Lemma}

{}
\def\1{1\!{\rm l}}

\newcommand{\leqa}{\lesssim}
\newcommand{\geqa}{\gtrsim}

\newcommand{\EM}{\ensuremath}

\newcommand{\al}{\alpha}
\newcommand{\be}{\beta}
\newcommand{\ze}{\zeta}

\newcommand{\ga}{\gamma}

\newcommand{\la}{\lambda}

\newcommand{\te}{\theta}
\newcommand{\ta}{\tau}
\newcommand{\veps}{\varepsilon}

\newcommand{\cA}{\EM{\mathcal{A}}}

\newcommand{\cC}{\EM{\mathcal{C}}}

\newcommand{\cM}{\EM{\mathcal{M}}}
\newcommand{\cN}{\EM{\mathcal{N}}}

\DeclareMathAlphabet{\mathpzc}{OT1}{pzc}{m}{it}

\newcommand{\RR}{\mathbb{R}}

\newcommand{\given}{\,|\,}

\newcommand{\R}{\mathds{R}}

\newcommand{\hal}{\hat\al}
\newcommand{\hala}{\hat\al_A}

\newcommand{\bi}{\begin{enumerate}[label=\roman*)]}
\newcommand{\ei}{\end{enumerate}}
\newcommand{\ba}{\begin{array}{rcl}}
\newcommand{\ea}{\end{array}}
\newcommand{\di}{\displaystyle}

\newcommand{\mockalph}[1]{}

\startlocaldefs
\theoremstyle{plain}
\newtheorem{thm}{Theorem}

\endlocaldefs

\begin{document}

\begin{frontmatter}
\title{Empirical Bayes analysis of spike and slab posterior distributions}
\runtitle{Spike and slab empirical Bayes}
\thankstext{T1}{Work partly supported by the grant ANR-17-CE40-0001-01
of the French National Research Agency ANR (project BASICS)}

\begin{aug}
\author{\fnms{Isma\"el} \snm{Castillo}
\ead[label=e1]{ismael.castillo@upmc.fr}}
\and
\author{\fnms{Romain} \snm{Mismer}
\ead[label=e2]{rmismer@free.fr}}

\address{Sorbonne Unversit\'e and Universit\'e Paris Diderot\\
Laboratoire Probabilit\'es, Statistique et et Mod\'elisation\\
UMR 8001\\
4, place Jussieu\\
75005 Paris, France\\
\printead{e1,e2}}

\runauthor{I. Castillo and R. Mismer}
\runtitle{Spike and Slab posterior convergence}

\affiliation{}

\end{aug}

\begin{abstract}
In the sparse normal means model, convergence of the Bayesian posterior distribution associated to spike and slab prior distributions is considered. The key sparsity hyperparameter is calibrated via marginal maximum likelihood empirical Bayes. The plug-in posterior squared--$L^2$ norm is shown to converge at the minimax rate for the euclidean norm for appropriate choices of spike and slab distributions. Possible choices include standard spike and slab with heavy tailed slab, and the spike and slab LASSO of Ro{\v c}kov\'a and George with heavy tailed slab. Surprisingly, the popular Laplace slab is shown to lead to a suboptimal rate for the empirical Bayes posterior itself. This provides a striking example where 
convergence of aspects of the empirical Bayes posterior such as the posterior mean or median does not entail convergence of the complete empirical Bayes posterior itself.
\end{abstract}




\end{frontmatter}


\section{Introduction}


In the sparse normal means model, one observes a sequence $X=(X_1,\ldots,X_n)$
\begin{equation} \label{model}
X_i = \theta_i + \veps_i,\quad i=1,\ldots,n,
\end{equation}
with $\theta=(\theta_1,\ldots,\theta_n) \in\RR^n$ and $\veps_1,\ldots,\veps_n$ 
i.i.d. $\cN(0,1)$. Given $\te$, the distribution of $X$ is a product of Gaussians and is  denoted by $P_\te$. Further, one assumes that the `true' vector $\theta_0$ belongs to 
\begin{equation*} 
 \ell_0[s_n] = \left\{\te\in\RR^n,\ \#\{i:\ \te_i\neq0\}\le s_n \right\},
\end{equation*}
the set of vectors that have at most $s_n$ nonzero coordinates, where $0\le s_n\le n$. A typical {\em sparsity} assumption is that $s_n$ is a sequence that may grow with $n$  but is `small' compared to $n$ (e.g. in the asymptotics 
$n\to\infty$, one typically assumes $s_n/n=o(1)$ and $s_n\to\infty$). A natural problem is that of estimating  $\te$ with respect to the euclidean loss $\|\te-\te'\|^2=\sum_{i=1}^n(\te_i-\te_i')^2$. A benchmark is given by the minimax rate for this loss over the class of sparse vectors $\ell_0[s_n]$. Denoting 
\[ r_n:=2s_n\log(n/s_n),\]
 \cite{djhs92} show that the minimax rate equals $(1+o(1))r_n$ as $n\to\infty$.

Taking a Bayesian approach, one of the simplest and arguably most natural classes of prior distributions in this setting is given by so-called {\em spike and slab} distributions,
\[ \te \, \sim \, \bigotimes_{i=1}^n\, (1-\al)\delta_0 + \al G, \]
where $\delta_0$ denotes the Dirac mass at $0$, the distribution $G$ has density $\ga$ with respect to Lebesgue measure and $\al$ belongs to $[0,1]$. These priors were introduced and advocated in a number of  papers, including \cite{mitchellbeauchamp, george, georgefoster, yuanlin}. One important point is the calibration of the tuning parameter $\al$, which can be done in a number of ways, including: deterministic $n$-dependent choice, data-dependent choice based on a preliminary estimate $\hat\al$, fully Bayesian choice based on a prior distribution on $\al$. Studying the behaviour of the posterior distributions in sparse settings is currently the object of a lot of activity. A brief (and by far not exhaustive) overview of recent works is given below. Given a prior distribution $\Pi$ on $\te$, and interpreting $P_\te$ as the law of $X$ given $\te$, one forms the posterior distribution $\Pi[\cdot\given X]$ which is the law of $\te$ given $X$. The frequentist analysis of the posterior distribution consists in the study of the convergence of $\Pi[\cdot\given X]$ in probability under $P_{\te_0}$, thus assuming that the data has actually been generated from some `true' parameter $\te_0$.

In the present paper, we follow this path and are more particularly interested in obtaining a uniform bound on the posterior squared $L^2$-moment of the order of the optimal minimax rate, that is in proving, with $C$ a large enough constant,
\begin{equation} \label{goal}
 \sup_{\te_0\in\ell_0[s_n]} E_{\te_0} \int \|\te-\te_0\|^2 d\Pi(\te\given X)
\le C r_n
\end{equation} 
for $\Pi$ a prior distribution constructed using a spike and slab approach, whose prior parameters may be calibrated using the data, that is following an empirical Bayes method. 
This is of interest for at least three reasons
\begin{itemize}
\item this provides adaptive convergence rates for the entire posterior distribution, using a fully data-driven procedure. This is more than obtaining convergence of aspects of the posterior such that posterior mean or mode, and in fact may require different conditions on the prior, as we shall see below. 
\item the inequality \eqref{goal} automatically implies convergence of several commonly used point estimators derived from the posterior $\Pi[\cdot\given X]$: it  implies convergence at rate $Cr_n$ of the posterior mean $\int\te d\Pi(\te\given X)$ (using Jensen's inequality, see e.g. \cite{cv12}), but also of the coordinatewise posterior median (see the supplement of \cite{cv12} for details) and in fact of any fixed posterior coordinatewise quantile, for instance the quantile $1/4$ of $\Pi[\cdot\given X]$. It also implies, using  Tchebychev's inequality, convergence of the posterior distribution at rate $M_nr_n$ for $\|\cdot\|^2$ as in \eqref{cvm} below with $M=M_n$, for any $M_n\to\infty$.  
\item knowing \eqref{goal} is a first step towards results for {\em uncertainty quantification}, in particular for the study of certain {\em credible sets}. Indeed, 
\eqref{goal} suggests a natural way to build such a set, that is 
 $\cC\subset\RR^n$ with $\Pi[\cC\given X] \ge 1-\al$ for a given $\al\in(0,1)$. Namely,  define $\cC=\{\te:\ \|\te-\bar\te\|^2\le r_X\}$, with $\bar\te$ the posterior mean (or another suitable point estimate of $\theta$) and $r_X$ a large enough multiple of the $(1-\al)$--quantile of $\int\|\te-\bar\te\|^2d\Pi(\te\given X)$. 
\end{itemize} 
The present work is the first of a series of papers where we study  aspects of inference using spike and slab prior distributions. In particular, based on the present results, the behaviour of the previously mentioned credible sets is studied in the forthcoming paper \cite{cs17}.

{\em Previous results on frequentist analysis of spike and slab type priors.}
In a seminal paper, Johnstone and Silverman \cite{js04} considered estimation of $\te$ using spike and slab priors combined with an empirical Bayes method for choosing $\al$. They chose $\al=\hat\al$ based on a marginal maximum likelihood approach to be described in more details below. Denoting $\hat\te$ the associated posterior median (or posterior mean), \cite{js04}
established that
\[ \sup_{\te_0\in\ell_0[s_n]} E_{\te_0}\|\hat\te - \te_0\|^2 \le Cr_n, \]
thereby proving minimaxity up to a constant of this estimator over $\ell_0[s_n]$. The estimator is adaptive, as the knowledge of $s_n$ is not required in its construction.

In \cite{cv12}, convergence of the posterior distribution is studied in the case $\al$ is given a prior distribution. If $\al\sim \text{Beta}(1,n+1)$, $\Pi$ is the corresponding hierarchical prior, and $\Pi[\cdot\given X]$ the associated posterior distribution, it is established in \cite{cv12}  that for large enough $M$, as $n\to\infty$,
\begin{equation} \label{cvm}
 \sup_{\te_0\in\ell_0[s_n]}  
 E_{\te_0} \Pi[\|\te-\te_0\|^2 \le M r_n \given X] \to 1.
\end{equation} 
In \cite{martin2014}, Martin and Walker use a fractional likelihood approach to construct a certain empirical Bayes spike and slab prior, where the idea is to reweight the standard spike and slab prior by a power of the likelihood.  They derive rate-optimal concentration results for the corresponding posterior distribution and  posterior mean. 

A related class of prior distributions recently put forward by Ro{\v c}kov\'a \cite{rockova17} and Ro{\v c}kov\'a and George \cite{rockovageorge17},  is given by 
\[ \te \, \sim \, \bigotimes_{i=1}^n \, (1-\al)G_0 + \al G_1, \] 
where both distributions $G_0, G_1$ have densities with respect to Lebesgue measure. The authors in particular consider the choices $G_0=\text{Lap}(\la_0)$ and $G_1=\text{Lap}(\la_1)$, where $\text{Lap}(\la)$ denotes the Laplace (double-exponential) distribution.  Taking $\la_0$ large enough enables one to mimic the spike of the standard spike and slab prior, and the fact that both $G_0, G_1$ are continuous distributions offers some computational advantages, especially when working with the posterior mode. One can also note that the posterior mode when $\al=1$ leads to the standard LASSO estimator. For this reason, the authors in \cite{rockova17, rockovageorge17} call this prior the {\em spike and slab LASSO} prior. It is shown in \cite{rockova17}, Theorem 5.2 and corollaries, that a certain deterministic $n$-dependent choice of $\al, \la_0, \la_1$ (but independent on the unknown $s_n$)  leads to posterior convergence at near-optimal rate $s_n\log{n}$, while putting a prior on $\al$ can yield (\cite{rockova17}, Theorem 5.4) the minimax rate for the posterior, if  a certain condition on the strength of the true non-zero coefficiencents of $\te_0$ is verified.

{\em Other priors and related work.}
We briefly review other options to induce sparsity using a Bayesian approach. One option considered in \cite{cv12} is first to draw a subset $S\subset\{1,\ldots,n\}$ at random and then to draw nonzero coordinates on this subset only.
 That is, sample first a dimension $k\in\{0,\ldots,n\}$ at random according to some prior $\pi$. Given $k$, sample $S$ uniformly at random over subsets of size $k$ and finally set
\begin{align*}  
\te_i & \sim G \quad i \in S\\
\te_i & = 0 \quad i \notin S.
\end{align*}  
Under the assumption that the prior $\pi$ on $k$ is of the form, referred to as the complexity prior,
\begin{equation} \label{dimcomp}
 \pi(k) = c e^{-ak\log(nb/k)},
\end{equation}
\cite{cv12} show that 
under this prior, both \eqref{cvm} and \eqref{goal} are satisfied. However, such a `strong' prior on the dimension is not necessary at least for \eqref{cvm} to hold: it can be checked for instance, for $\pi$ the prior on dimension induced by the spike and slab prior on $\te$ with $\al\sim \text{Beta}(1,n+1)$, that $\pi(s_n)\asymp \exp(-cs_n)\gg \exp(-cs_n\log(n/s_n))$.   
So in a sense the complexity prior `penalises slightly more than necessary'. 

Another popular way to induce sparsity is via the so-called {\em horseshoe} prior, which draws a $\te$ from a continous distribution which is itself a mixture. As established in \cite{vsv17}--\cite{vsvuq17} the horseshoe yields the nearly-optimal rate $s_n\log{n}$ uniformly over the whole space $\ell_0[s_n]$, up again to the correct form of the logarithmic factor. In a different spirit but still without using Dirac masses at $0$, the paper \cite{jiangzhang09} shows that, remarkably, it is also possible to adopt an empirical Bayes approach on the entire unknown distribution function $F$ of the vector $\te$, interpreting $\te$ as sampled from a certain distribution, and the authors derive oracle results over $\ell^p, p>0,$ balls for the plug-in posterior mean (not including the case $p=0$ though). We also note the interesting work \cite{salometal} that investigates necessary and sufficient conditions for sparse continuous priors to be rate-optimal. However the latter is for a fixed regularity parameter $s_n$, while the results decribed in Section \ref{sec-main} (in particularity the suboptimality phenomenon, but also upper-bounds using the empirical Bayes approach) are related to adaptation.

Using complexity--type priors on the number of non-zero coordinates, Belitser and co-authors \cite{babbel10}--\cite{belnur15} consider Gaussian priors on non-zero coefficients, with a recentering of the posterior mean at the observation $X_i$-- for those coordinates $i$ that are selected-- to adjust for overshrinkage. In \cite{belnur15}, oracle results for the corresponding posterior are derived, that in particular imply convergence at the minimax rate up to constant over $\ell_0[s_n]$, and the authors also derive results on uncertainty quantification by studying the frequentist coverage of credible sets using their procedure.

For further references on the topic, in particular about relationships between spike and slab priors and absolutely continuous counterparts such as the horseshoe or the spike and slab LASSO, we refer to the paper \cite{vsvuq17} and its discussion by several authors of the previously mentioned works.

{\em Overview of results and outline.}  This paper obtains the following results.
\begin{enumerate}
\item For the spike and slab prior, in Section \ref{sec-subop} we establish lower bound results that show that the popular Laplace slab yields suboptimal rates when the complete empirical Bayes posterior is considered.
\item In Sections \ref{sec-opt} and \ref{sec-mod}, we establish rate-optimal results for the posterior squared $L^2$--moment for the usual spike and slab with a Cauchy slab, when the prior hyperparameter is chosen via a marginal maximum likelihood method. 
\item In Section \ref{sec-ssl}, the spike and slab LASSO prior is considered and we provide a near-optimal adaptive rate for the corresponding complete empirical Bayes posterior distribution. 
\end{enumerate}
Section \ref{sec-main} introduces the framework, notation, and the main results, ending with a brief simulation study in  Section \ref{sec-sim} and discussion. Section \ref{sec-pr} gathers the proofs of the lower-bound results as well as upper-bounds on the spike and slab prior. Technical lemmas for the spike and slab prior can be found in Section \ref{sec-tec}, while Sections \ref{sec-sslpr}--\ref{seclasso} contain the proof of the result for the spike and slab LASSO prior.

For real-valued functions $f, g$, we write $f\leqa g$ if there exists a universal constant $C$ such that $f(x)\le Cg(x)$, and $f \geqa g$ is defined similarly. When $x$ is a positive real number or an integer, we write $f(x)\asymp g(x)$  if  there exists positive constants $c,  C, D$ such that for $x\ge D$, we have $cf(x)\le g(x)\le Cf(x)$. For reals $a,b$, one denotes $a\wedge b=\min(a,b)$ and $a\vee b=\max(a,b)$.

\section{Framework and main results} 
\label{sec-main}
 
\subsection{Empirical Bayes estimation with spike and slab prior}
\label{sec-not}

In the setting of model \eqref{model}, the spike and slab prior on $\te$ with fixed parameter $\alpha\in[0,1]$ is
\begin{equation} \label{priorsas}
 \Pi_\alpha \sim \otimes_{i=1}^n (1-\al)\delta_0 + \al G(\cdot), 
\end{equation} 
where   $G$ is a given probability measure on $\RR$. 
We consider the following choices 
\[ G = 
\begin{cases}
& \text{Lap}(1) \\
or\\
 & \text{Cauchy}(1)
\end{cases} 
  \]
where Lap$(\la)$ denotes the Laplace (double exponential) distribution with parameter $\la$ and Cauchy$(1)$ the standard Cauchy distribution.
Different choices of parameters and prior distributions are possible (a brief discussion is included below) but for clarity of exposition we stick to these common distributions. In the sequel $\ga$ denotes the density of $G$ with respect to Lebesgue measure.

By Bayes' formula the posterior distribution under \eqref{model} and \eqref{priorsas} with fixed $\al\in[0,1]$ is 
\begin{equation} \label{post}
\Pi_\alpha[\cdot\given X] 
\sim \otimes_{i=1}^n (1-a(X_i))\delta_0 + a(X_i) G_{X_i}(\cdot),
\end{equation}
where, denoting by $\phi$ the standard normal density and 
$g(x)=\phi*G(x)=\int \phi(x-u)dG(u)$ the convolution of $\phi$ and $G$ at point $x\in\RR$,  the posterior weight $a(X_i)$ is given by, for any $i$,
\begin{equation} \label{postwei}
 a(X_i) = a_\al(X_i)=\frac{\al g(X_i)}{(1-\alpha)\phi(X_i) + \alpha g(X_i)}. 
\end{equation} 
The distribution $G_{X_i}$ has density
\begin{equation} \label{postdc}
\ga_{X_i}(\cdot) := \frac{\phi(X_i-\cdot) \ga(\cdot)}{g(X_i)}
\end{equation}
with respect to Lebesgue measure on $\RR$. The behaviour of the posterior distribution $\Pi_\al[\cdot\given X]$ heavily depends on the choices of the smoothing parameters $\al$ and $\ga$. 
It turns out that some aspects of this distribution are thresholding-type estimators, as established in \cite{js04}. 

{\em Posterior median and threshold $t(\al)$.} The posterior median $\hat\te^{med}_\al(X_i)$ of the $i$th coordinate has a thresholding property:  there exists $t(\al)>0$ such that $\hat\te^{med}_\al(X_i)=0$ if and only if $|X_i|\le t(\al)$. 
A default choice can be $\al=1/n$; one can check that this leads to a posterior median behaving similarly as a hard thresholding estimator with threshold $\sqrt{2\log n}$. One can significantly improve on this default choice by taking a well-chosen data-dependent $\al$. 

In order to choose $\al$, in this paper we follow the empirical Bayes method proposed in \cite{js04}. The idea is to estimate $\al$ by maximising the marginal likelihood in $\al$ in the Bayesian model, which is the density of $\al\given X$. The log-marginal likelihood in $\alpha$ can be written as
\begin{equation} \label{mli}
\ell(\al) = \ell_n(\al;X)= \sum_{i=1}^n \log( (1-\al)\phi(X_i) + \al g(X_i)). 
\end{equation}
Let $\hal$ be defined as the maximiser of the log-marginal likelihood 
\begin{equation} \label{defhal}
\hal = \underset{\al\in \cA_n}{\text{argmax}}\ \ell_n(\al;X),
\end{equation}
where the maximisation is restricted to 
$\cA_n=[\al_n,1]$, 
with $\al_n$ defined by
\[ t(\al_n)=\sqrt{2\log{n}}. \]
The reason for this restriction is that one does not need to take $\al$ smaller than $\al_n$, which would correspond to a choice of $\al$ `more conservative' than hard-thresholding at threshold level $\sqrt{2\log{n}}$.  

In \cite{js04}, Johnstone and Silverman prove that the posterior median $\hat\al^{med}(X_i)$ has remarkable optimality properties, for many choices of the slab density $\ga$. For $\ga$ with tails `at least as heavy as' the Laplace distribution, then this point estimator converges at the minimax rate over $\ell_0[s_n]$. More precisely, it follows from Theorem 1 in \cite{js04} that there exists  constants $C, c_0, c_1$ such that if
\begin{equation} \label{techsn}
c_1\log^2 n \le  s_n\le c_0 n,
\end{equation}
then the posterior median $\hat\te^{med}_{\hat\al}=(\hat\te^{med}_{\hat\al}(X_i))_{1\le i\le n}$ is rate optimal
\begin{equation} \label{postmedian}
 \sup_{\te\in \ell_0[s_n]} E_{\te}\|\hat\te^{med}_{\hat\al}-\te\|^2 \le Cs_n\log(n/s_n).
\end{equation} 
One can actually remove the lower bound in condition \eqref{techsn} -- see Theorem 2  in \cite{js04}  -- by a more complicated choice of $\hat\al$, for which $\hat\al$ in \eqref{defhal} is replaced by a smaller value if the empirical Bayes estimate is close to $\al_n$ given by   $t(\al_n)=\sqrt{2\log{n}}$. In the present paper for simplicity of exposition we first work under the condition \eqref{techsn}. In Section \ref{sec-mod} below, we show that the lower bound part of the condition can be removed when working with the modified estimator as in \cite{js04}.

{\em Plug-in posterior distribution.} 
The posterior we consider in this paper is $\Pi_{\hat\al}[\cdot\given X]$, that is the distribution given by \eqref{post}, where $\al$ has been replaced by its empirical Bayes (EB) estimate $\hat\al$ given by \eqref{defhal}. This posterior is called complete EB posterior in the sequel. The value $\hat\al$ is easily found numerically, as implemented in the R package {\tt EbayesThresh}, see \cite{js05}.
As noted in \cite{js04}, the posterior median $\hat\al^{med}(X_i)$ displays excellent behaviour in simulations.  However, the entire posterior distribution $\Pi_{\hat\al}[\cdot\given X]$ has not been studied so far. It turns out that the behaviour of the posterior median does not always reflect the behaviour of the complete posterior, as is seen in the next subsection.

\newpage
\subsection{Suboptimality of the Laplace slab for the complete EB posterior distribution} 
\label{sec-subop}

\begin{thm} \label{negres}
Let $\Pi_\al$ be the spike and slab prior distribution \eqref{priorsas} with slab distribution $G$ equal to the Laplace distribution Lap$(1)$. Let $\Pi_{\hat \al}[\cdot\given X]$ be the corresponding plug-in posterior distribution given by \eqref{post}, with $\hat \al$ chosen by the empirical Bayes procedure \eqref{defhal}. 
There exist $D>0$, $N_0>0$, and $c_0>0$ such that, for any $n\ge N_0$ and any $s_n$ with  $1\le s_n \leq c_0 n$, there exists $\te_0\in \ell_0[s_n]$ such that, 
\[ E_{\te_0} \int \|\te-\te_0\|^2 d\Pi_{\hat\al}[\te\given X] \ge D s_n 
e^{\sqrt{\log{(n/s_n)}}}.
\]
\end{thm}

Theorem \ref{negres} implies that taking a Laplace slab leads to a suboptimal convergence rate in terms of the posterior squared $L^2$--moment. This result is surprising at first, as we know by \eqref{postmedian} that the posterior median converges at optimal rate $r_n$. The posterior mean also converges at rate $r_n$ uniformly over $\ell_0[s_n]$, by Theorem 1 of \cite{js04}. So at first sight it would be quite natural to expect that so does the posterior second moment. 

One can naturally ask whether the suboptimality result from Theorem \ref{negres} could come from considering an integrated $L^2$--moment, instead of simply asking for a posterior convergence result in probability, as is standard in the posterior rates literature following \cite{ggv}. We now derive a stronger result than Theorem \ref{negres} under the mild condition $s_n \gtrsim \log^2 n$. The fact that the result is stronger follows from  bounding from below the integral in the display of Theorem \ref{negres} by the integral restricted to the set where $\|\te-\te_0\|^2$ is larger than the target lower bound rate.
\begin{thm} \label{negres2}
Under the same notation as in Theorem \ref{negres}, if $\Pi_\al$ is a spike and slab distribution with as slab $G$ the Laplace distribution, there exists $m>0$  such that for any $s_n$ with $s_n/n\to 0$ and $\log^2{n}=O(s_n)$ as $n\to\infty$, there exists $\te_0\in\ell_0[s_n]$ such that, as $n\to\infty$, 
\[ E_{\te_0} \Pi_{\hat\al} \left[\, \|\te-\te_0\|^2  \le m s_n 
e^{\sqrt{2\log{(n/s_n)}}}\given X\,\right] =o(1). 
\]
\end{thm}
Theorem \ref{negres2}, by  providing a lower bound in the spirit of \cite{ic08},  shows that the answer to the above question is negative, and for a Laplace slab, the plug-in posterior  $\Pi_{\hat \al}[\cdot\given X]$ does not converge at minimax rate uniformly over $\ell_0[s_n]$. 

Note that the suboptimality occuring here does not result from an artificially constructed example (we work under exactly the same framework as \cite{js04}) and that  this has important (negative)  consequences for construction of credible sets. Due to the rate-suboptimality of the EB Laplace-posterior, typical credible sets derived from it (such as, e.g., taking quantiles of a recentered posterior second moment) will inherit the suboptimality in terms of their diameter, and therefore will not be of optimal size. Fortunately, it is still possible to achieve optimal rates for certain spike and slab EB posteriors: the previous phenomenon indeed disappears if the tails of the slab in the prior distribution are heavy enough, as seen in the next subsection.

\subsection{Optimal posterior convergence rate for the EB spike and Cauchy slab}
\label{sec-opt}

The next result considers Cauchy tails, although other examples can be covered, as discussed below. 
In the sequel, we abbreviate by SAS prior a spike and slab prior with Cauchy slab. 

\begin{thm} \label{thm-risk}
Let $\Pi_\al$ be the SAS prior distribution \eqref{priorsas} with slab distribution $G$ equal to the standard Cauchy distribution. Let $\Pi_{\hat \al}[\cdot\given X]$ be the corresponding plug-in posterior distribution given by \eqref{post}, with $\hat \al$ chosen by the empirical Bayes procedure \eqref{defhal}. There exist $C>0$, $N_0>0$, and $c_0, c_1>0$ such that, for any $n\ge N_0$, for any $s_n$ such that \eqref{techsn} is satisfied for such $c_0,c_1$, 
\[ \sup_{\te_0\in \ell_0[s_n]} E_{\te_0} \int \|\te-\te_0\|^2 d\Pi_{\hat\al}(\te\given X) 
\le Cr_n.\]
If one only assumes $s_n\le c_0n$ in \eqref{techsn}, then the last statement holds with the bound $Cr_n$ replaced by $Cr_n+C\log^3{n}$. 
\end{thm}
Theorem \ref{thm-risk} confirms that the empirical Bayes plug-in posterior, with $\hat\alpha$ chosen by marginal maximum likelihood, converges at optimal rate with precise logarithmic factor, at least under the mild condition \eqref{techsn}, if tails of the slab distribution are heavy enough. Inspection of the proof of Theorem \ref{thm-risk} reveals that any slab density $\ga$ with tails of the order $x^{-1-\delta}$ with $\delta\in(0,2)$ gives the same result. Sensibility to the tails, in particular in view of posterior convergence in terms of $d_q$-distances, will be further investigated in \cite{cs17}. 

We note that the horseshoe prior on $\te$ considered in \cite{vsv17}--\cite{vsvuq17}  also has Cauchy-like tails, which seems to confirm that for empirical Bayes--calibrated (product--type) sparse priors, heavy tails are important to ensure optimal or near-optimal behaviour, see also the discussion \cite{icdisc17}. 

The lower bound in condition \eqref{techsn} is specific to the estimate $\hat \al$. Note that in the very sparse regime where $s_n\le c_1\log^2{n}$, the rate is no more than $C\log^3{n}$, thus missing the minimax rate by at most a logarithmic factor. This lower bound on $s_n$ can be removed and the minimax rate obtained over the whole range of sparsities $s_n$ if one modifies slightly $\hat\alpha$, where the estimator is changed if  $\hat\alpha$ is too close to the lower boundary of the maximisation interval, see Section \ref{sec-mod}.

\subsection{Posterior convergence for the EB spike and slab LASSO}
\label{sec-ssl}

Now consider the following prior on $\te$ with fixed parameter $\alpha\in[0,1]$ 
\begin{equation} \label{priorlasso}
 \Pi_\alpha \sim \otimes_{i=1}^n (1-\al)G_0(\cdot) + \al G_1(\cdot), 
\end{equation} 
where for $k=0,1$,  $G_k$ is given by
\[ G_0= \text{Lap}(\la_0),\quad G_1 = 
\begin{cases}
& \text{Lap}(\la_1) \\
\text{ or } \\
& \text{Cauchy}(1/\la_1),
\end{cases} \]
which leads to the spike and slab LASSO prior of \cite{rockovageorge17} in the case of a Laplace $G_1$, and to a heavy-tailed variant of the spike and slab LASSO  if  $G_1$ is Cauchy$(1/\la_1)$, that is if its density is $\ga_1(x)=(\la_1/\pi)(1+\la_1^2x^2)^{-1}$.  In this setting $\ga_0, \ga_1$ denote the densities of $G_0, G_1$ with respect to Lebesgue measure. We call {\em SSL prior} a spike and slab LASSO prior with Cauchy slab.

By Bayes' formula the posterior distribution under \eqref{model} and \eqref{priorlasso} with fixed $\al\in[0,1]$ is 
\begin{equation} \label{post2}
\Pi_\alpha[\cdot\given X] 
\sim \otimes_{i=1}^n (1-a(X_i))G_{0,X_i}(\cdot) + a(X_i) G_{1,X_i}(\cdot),
\end{equation}
where $g_k(x)=\phi*G_k(x)=\int \phi(x-u)dG_k(u)$ is the convolution of $\phi$ and $G_k$ at point $x\in\RR$ for $k=0,1$,  the posterior weight $a(X_i)$ is defined through the function $a(\cdot)$ given by
\[ a(x) = a_\al(x) =\frac{\al g_1(x)}{(1-\alpha)g_0(x) + \alpha g_1(x)}, \]
and if $G_k$ has density $\gamma_k$ with respect to Lebesgue measure, the distribution $G_{k,X_i}$ has density
\begin{align*}
\ga_{k,X_i}(\cdot) & := \frac{\phi(X_i-\cdot) \ga_k(\cdot)}{g_k(X_i)}.
\end{align*}
In slight abuse of notation, we keep the same notation in the case of the SSL prior for quantities such as $a(x)$ or $\hat\al$ below, as it will always be clear from the context which prior we work with. 

We consider the following specific choices for the constants $\la_0, \la_1$
\begin{equation} \label{param}
\begin{cases}
\ \la_0\ \  = & L_0 n, \qquad \quad  L_0 =5\sqrt{2\pi}, \\
\ \la_1\ \ = &  L_1,\ \  \qquad \quad L_1 = 0.05. 
\end{cases}
\end{equation}
The choice of the constants $L_0, L_1$ is mostly for technical convenience, and is similar to that of, e.g. Corollary 5.2 in \cite{rockova17}. Any other constant $L_0$ (resp. $L_1$) larger (resp. smaller) than the above value also works for the following result. The above numerical values may not be optimal.

Let $\hal$ be defined as the maximiser of the log-marginal likelihood, 
\begin{equation} \label{defhalL}
\hal = \underset{\al\in [\mathcal{C}\log n/n,1]}{\text{argmax}}\ \ell_n(\al;X),
\end{equation}
for $\cC=\cC_0(\ga_0,\ga_1)$ a large enough constant to be chosen below (this ensures that $\hal$ belongs to an interval on which we can verify that $\beta$ is increasing, see \eqref{after}). This time we do not have access to the threshold $t$, since for the SSL prior the posterior median is not a threshold estimator, so here $\mathcal{C}\log n/n$ plays the role of an approximated version of $\al_n$ in \eqref{defhal}. 

\begin{thm} \label{thm-lasso}
Let $\Pi_\al$ be the SSL prior distribution \eqref{priorlasso} with Cauchy slab and parameters $(\la_0,\la_1)$ given by \eqref{param}. 
Let $\Pi_{\hat \al}[\cdot\given X]$ be the corresponding plug-in posterior distribution given by \eqref{post2}, with $\hat \al$ chosen by the empirical Bayes procedure \eqref{defhalL}. There exist $C>0$, $N_0>0$, and $c_0, c_1>0$ such that, for any $n\ge N_0$, for any $s_n$ such that \eqref{techsn} is satisfied for such $c_0,c_1$, then 
\[ \sup_{\te_0\in \ell_0[s_n]} E_{\te_0} \int \|\te-\te_0\|^2 d\Pi_{\hat\al}(\te\given X) 
\le Cs_n\log{n}.\]
If one only assumes $s_n\le c_0n$ in \eqref{techsn}, then the last bound holds with $Cs_n\log{n}$ replaced by $C(s_n\log{n}+\log^3{n})$. 
\end{thm}

This result is an SSL version of Theorem \ref{thm-risk}. It shows that a spike and slab LASSO prior with heavy-tailed slab distribution and empirical Bayes choice of the weight parameter leads to a nearly optimal contraction rate for the entire posterior distribution. Hence it provides a theoretical guarantee of a fully data-driven procedure of calibration of the smoothing parameter in SSL priors. 

\subsection{A brief numerical study} \label{sec-sim}

Theorems \ref{negres}--\ref{negres2} imply that the posterior distribution for the spike and slab  prior and Laplace$(1)$ slab does not converge at optimal rate and the discrepancy between the actual rate and the minimax rate for some `bad' $\te_0$s is at least of order
\[ R_n = \frac{\exp\left(\sqrt{2\log(n/s_n)}\right)}{\log(n/s_n)},\]
up to a multiplicative constant factor, as both lower and upper bounds are up to a constant. Note that $R_n$ grows more slowly than a polynomial in $n/s_n$, so the sub-optimality effect will typically be only visible for quite large values of $n/s_n$. For instance, if $n=10^4$ and $s_n=10$, one has $R_n\approx 6$, which is quite small given that an extra multiplicative constant is also involved. 
 
For the present simulation study we took $n=10^7$, $s_n=10$, for which $R_n\approx 13.9$, and the non-zero values of $\te_0$ equal to $\{2\log(n/s_n)\}^{1/2}$, as the lower bound proof of Theorems \ref{negres}--\ref{negres2} suggests. We computed $\hat\al$ using the package {\tt EBayesThresh} of Johnstone and Silverman \cite{js05} and computed $\int \|\te-\te_0\|_2^2d\Pi_{\hat\al}(X)$ using its explicit expression, which can be obtained in closed form for a Laplace slab, with similar computations as in \cite{js05}, Section 6.3. We then took the empirical average over $100$ repetitions to estimate the target expectation $R_2:=E_{\te_0}\int \|\te-\te_0\|_2^2d\Pi_{\hat\al}(X)$. We first took $\ga=\text{Lap}(1)$ a standard Laplace slab and obtained $\hat{R}_2\approx 1110$. For comparison, we computed the empirical quadratic risk $\hat R_{mean}$ for the posterior mean (approximating $E_{\te_0}\|\hat\te^{mean}-\te_0\|^2$) and $\hat R_{median}$ the posterior median of the same posterior, obtaining $\hat R_{mean}\approx 158$ and $\hat R_{median}\approx 167$. So, in this case $\hat R_2$ is already $6$ to $7$ times larger than the risk of either mean or median. 

 To further illustrate the `blow-up' in the rate for the posterior second moment $R_2$, we took a Laplace slab Lap$(a)$ with inverse-scale parameter $a$, for which the numerator in the definition of $R_n$ becomes $\exp\{a\sqrt{2\log(n/s_n)}\}$ (let us also note that the multiplicative constant we refer to above also depends on $a$). The same simulation experiment as above was conducted, with the standard Laplace slab replaced by a Lap$(a)$ slab, for different values of $a$. The numerical results are presented in Table \ref{tableres}, which feature a noticeable increase in the second moment $\hat{R}_2$, while the risks of posterior mean and median stay around the same value, as expected.  
 
\begin{table} [!ht]
\begin{center}
\begin{tabular}{|l|c|c|c|c|c|c|c|}
\hline
$a$ & 0.5 & 1 & 1.5 & 2 & 2.5 & 3 & 3.5   \\
\hline
Second moment & 394  & 1110 &  2847 & 5716 & 8093 &  16530 & 34791  \\
\hline
Median & 173 & 167 & 169 & 174 & 185 & 209 & 219  \\
\hline
Mean    & 157 & 158 & 166 & 172 & 182 & 224 & 336  \\
\hline
\end{tabular}
\end{center}
\caption{Empirical risks $\hat{R}_2, \hat{R}_{med}, \hat{R}_{mean}$ for Laplace slabs Lap$(a)$ and $a\in[0.5,3.5]$} \label{tableres}
\end{table}

We also performed the same experiments for the quasi-Cauchy slab prior introduced in \cite{js04}-\cite{js05} (it is very close to the standard Cauchy slab -- in particular it has the same Cauchy tails -- but more convenient from the numerical perspective, see \cite{js05}, Section 6.4). We found $\hat{R}^{median}\approx 192$, $\hat{R}^{mean}\approx 191$ for the posterior mean and $\hat{R}_2\approx 287$ for the posterior second moment. This time, as expected, the posterior second moment is not far from the two other risks.

\subsection{Modified empirical Bayes estimator} \label{sec-mod}

For $n\ge 3$ and $A \ge 0$, let us set $t_n^2=2\log{n} - 5\log\log{n}$ and $t_A=\sqrt{2(1+A)\log{n}}$. 
For $\Pi_\al$ the SAS prior with a Cauchy slab, let as before $t(\al)$ be the posterior median threshold for fixed $\al$. It is not hard to check that $t(\cdot)$ is continuous and strictly decreasing so has an inverse (see \cite{js04}, Section 5.3).  In a similar fashion as in \cite{js04}, Section 4,  let us introduce a modified empirical Bayes estimator as, for $A\ge 0$ and $\hat t:=t(\hat\al)$, $\al_A:=t^{-1}(t_A)$,
\begin{equation} \label{modal}
\hat\al_A =
\begin{cases}
\hat\al,\qquad &\text{if } \hat t\le t_n,\\
\al_A,\qquad &\text{if } \hat t> t_n.
\end{cases}
\end{equation}

\begin{thm} \label{thm-risk-mod}
Let $\Pi_\al$ be the SAS prior distribution  with slab distribution $G$ equal to the standard Cauchy distribution. For a fixed  $A>0$, let $\Pi_{\hat \al_A}[\cdot\given X]$ be the corresponding plug-in posterior distribution, with $\hat \al_A$ the {\em modified} estimator \eqref{modal}. There exist $C, c_0>0$, $N_0>0$, such that, for any $n\ge N_0$, for any $s_n$ such that $s_n\le c_0 n$,
\[ \sup_{\te_0\in \ell_0[s_n]} E_{\te_0} \int \|\te-\te_0\|^2 d\Pi_{\hat\al}(\te\given X) 
\le Cr_n.\]
\end{thm}
Theorem \ref{thm-risk-mod} shows that the plug-in SAS posterior distribution  using the modified estimator \eqref{modal}, $A>0$, and a Cauchy slab  attains the minimax rate of convergence $r_n$ even in the very sparse regime $s_n\leqa \log^2{n}$, for which the unmodified estimate of  Theorem \ref{thm-risk} 
may lose a logarithmic factor.  

\subsection{Discussion}

In this paper, we have developped a theory of empirical Bayes choice of the hyperparameter of spike and slab prior distributions. It extends the work of Johnstone and Silverman \cite{js04} in that here the complete EB posterior distribution is considered. One important message is that such a generalisation preserves optimal convergence rates at the condition of taking slab distributions with heavy enough tails. If the tails of the slab are only moderate (e.g. Laplace), then the complete EB posterior rate may be suboptimal. This is in contrast with the hierarchical case considered in \cite{cv12}, where a Laplace slab combined with a Beta distributed prior on $\al$ was shown to lead to an optimal posterior rate. On the one hand, the empirical Bayes method often leads to simpler or/and more easily tractable practical algorithms; on the other hand, we have illustrated here that the complete EB posterior may in some cases need slightly stronger conditions to conserve optimal theoretical guarantees. This phenomenon had not been pointed out so far in the literature, to the best of our knowledge. 
 
We also note that Theorem \ref{thm-risk} (or Theorem \ref{thm-risk-mod} if one allows for very sparse signals) enables one to recover the optimal form of the logarithmic factor $\log(n/s_n)$ in the minimax rate. This entails significant work, as one needs to control the empirical Bayes weight estimate $\hat \al$ both from above {\em and below}. 
This could work too in the  SSL setting of Theorem \ref{thm-lasso}, although this seems to need substantial extra technical work.


Looking at Theorems \ref{negres} and \ref{negres2}, it is natural to wonder why the Empirical Bayes approach fails for the Laplace slab where the full Bayes approach succeeds as seen in \cite{cv12} Theorem 2.2. The reason why the hierarchical Bayes version works also for $\ga$ Laplace is the extra penalty in model size induced by the hierarchical prior on dimension.  Indeed, in the full Bayes approach, the posterior distribution of $\al$ given $X$ has density 
\[ f_{\al\given X}(\al) \propto p(X\given \al)\pi(\al), \]
where $p(X\given \al)$ is the marginal density one maximises when considering the MMLE $\hat \al$. Hence adding a term $\log\pi(\al)$ for well-chosen $\pi$ -- for instance that arising from a $\text{Beta}(1,n+1)$ prior on $\alpha$ as considered in \cite{cv12} -- to the log-marginal likelihood one maximises  forces $\hat\al$ to concentrate on smaller values. For instance, in the present setting, one could consider a penalised log-marginal maximum likelihood, which would force the estimate $\hat \al$ to concentrate on slightly smaller values, which would allow one to avoid the extra  $e^{\sqrt{\log{n/s_n}}}$ term arising in Theorems \ref{negres}--\ref{negres2}.

The present work can also serve as a basis for constructing confidence regions using spike-and-slab posterior distributions. This question is considered in the forthcoming paper \cite{cs17}.


\section{Proofs for the spike and slab prior } \label{sec-pr}

Let us briefly outline the ingredients of the proofs to follow. For Theorems \ref{negres} and \ref{thm-risk}, our goal is to bound the expected posterior risk $R_n(\te_0)=E_{\te_0} \int \|\te-\te_0\|^2 d\Pi_{\hat\al}(\te\given X)$. 
There are three main tools. First,   after introducing  notation and basic bounds in Section \ref{sec-nota}, bounds on the posterior risk for fixed $\al$ are given in Section \ref{sec-fab}, as well as corresponding bounds for random $\al$. Let us note that the corresponding upper bounds are different from those obtained on the quadratic risk for the posterior median in \cite{js04} (and in fact, must be, in view of the negative result in Theorem \ref{negres}). Second, inequalities on moments of the score function are stated in Section \ref{sec-mom}. As a third tool, we obtain deviation inequalities on the location of $\hat \al$ in Section \ref{sec-hal}. One of the bounds sharpens the corresponding bound from \cite{js04} in case the signal belongs to the nearly-black class $\ell_0[s_n]$ which we assume here. 

Proofs of Theorems \ref{negres} and \ref{thm-risk} are given in Sections \ref{prn} and \ref{prrisk}. For Theorem \ref{thm-risk}, we also needed to slightly complete the proof of one of the inequalities on thresholds stated in \cite{js04}, see Lemma \ref{lem-lb3}. The proof of Theorem \ref{negres2}, which uses ideas from both previous proofs, is given in Section \ref{prn2}. Proofs of technical lemmas for the SAS prior are given in Section \ref{sec-tec}.

\subsection{Notation and tools for the SAS prior}\label{sec-nota}

{\em Expected posterior $L^2$--squared risk.} For a fixed weight $\al$, the posterior distribution of $\te$ is given by  \eqref{post}. On each coordinate, the mixing weight $a(X_i)$ is given by \eqref{postwei} and the density of the non-zero component $\ga_{X_i}$  by \eqref{postdc}. 
In the sequel we will obtain bounds on the following quantity, already for 
a given $\alpha\in[0,1]$, 
\[ \int \|\te-\te_0\|^2 d\Pi_\al(\te\given X) = \sum_{i=1}^n
\int (\te_i - \te_{0,i})^2 d\Pi_\al(\te_i\given X_i). \]
To do so, we study $r_2(\al,\mu,x):= \int (u - \mu)^2d\pi_\al(u\given x)$, where 
$\pi_\al(\cdot\given x) \sim (1-a(x))\delta_0 + a(x) \ga_x(\cdot).$
By definition
\begin{align*}
r_2(\al,\mu,x) = (1-a(x))\mu^2 + a(x)\int (u-\mu)^2\ga_x(u)du.
\end{align*}
This quantity is controlled by $a(x)$ and the term involving $\ga_x$. From the definition of $a(x)$, bounding the denominator from below by one of its two components, and using $a(x)\le 1$ yields, for any real $x$ and $\al\in [0,1]$,
\begin{align}
\al \frac{g}{g\vee\phi}(x)\le a(x)  \le 1 \wedge \frac{\al}{1-\al} \frac{g}{\phi}(x).
\label{sima}
\end{align}

{\em The marginal likelihood in $\al$.} By definition, the empirical Bayes estimate $\hat\al$ in \eqref{defhal} maximises the logarithm of the marginal likelihood in $\al$ in \eqref{mli}. In case the maximum is not taken at the boundary, $\hat\al$ is a zero of the derivative (score) of the previous likelihood. 
Its expression is $S(\al)=\sum_{i=1}^n \beta(X_i,\al)$, where following \cite{js04} we set, for $0\le\al\le 1$ and any real $x$,
\[ \beta(x,\al)=\frac{\be(x)}{1+\al\be(x)},\qquad  \beta(x) = \frac{g}{\phi}(x)-1. \]

The study of $\hat \al$ below uses in a crucial way the first two moments of $\beta(X_i,\al)$, so we introduce the corresponding notation next.  
Let $E_\ta$, for $\ta\in\RR^n$, denote the expectation under $\te_0=\ta$. Define
\begin{equation} \label{mtilde}
\tilde m(\al) = - E_0\beta(X,\al)
\end{equation}
and further denote 
\begin{align*} 
m_1(\ta,\al) & = E_\ta[\be(X,\al)] = \int_{-\infty}^{\infty} \be(t,\al)\phi(t-\tau)dt. \\
m_2(\ta,\al) & = E_\ta[\be(X,\al)^2]. 
\end{align*}

{\em The thresholds $\zeta(\al)$, $\tilde\ta(\al)$ and $t(\al)$.} Following \cite{js04}, we introduce several useful thresholds. From Lemma 1 in \cite{js04}, we know that $g/\phi$, and therefore $\be=g/\phi-1$, is a strictly increasing function on $\RR^+$. It is also continuous, so given $\al$, a pseudo-threshold $\zeta=\zeta(\al)$ can be defined by 
\begin{equation} \label{zeta}
 \beta(\zeta) = \frac1\al.
\end{equation}
Further one can also define $\tau(\al)$ as the solution in $x$ of 
\[ \Omega(x,\al):=\frac{a(x)}{1-a(x)}=\frac{\al}{1-\al}\frac{g}{\phi}(x)=1. \]
Equivalently, $a(\ta(\al))=1/2$. Also, $\be(\ta(\al))=\al^{-1}-2$ so $\ta(\al)\le \zeta(\al)$. 
Define $\al_0$ as $\ta(\al_0)=1$ and set 
\begin{equation} \label{tauti}                    
 \tilde\ta(\al) = \ta(\al\wedge \al_0).         
\end{equation}                                    
Recall from Section \ref{sec-main} that $t(\al)$ is the threshold associated to the posterior median for given $\al$. It is shown in  \cite{js04}, Lemma 3, that $t(\al)\le \zeta(\al)$. Finally, the following bound in terms of $\ta(\al)$, see \cite{js04} p. 1623, is also useful for large $x$,
\begin{equation}\label{postwb}
 1-a(x) \le 1\1_{|x|\le \tilde\ta(\al)}
 +   e^{-\frac12 (|x|-\tilde\ta(\al))^2}\1_{|x| > \tilde\ta(\al)}.
\end{equation}

\subsection{Posterior risk bounds} \label{sec-fab}

Recall the notation $r_2(\al,\mu,x)=\int (u-\mu)^2 d\Pi_\al(u)$. 

\begin{lem} \label{lemfia}
Let $\ga$ be the Cauchy or Laplace density. For any $x$ and $\al\in[0,1/2]$,
\begin{align*}  
 r_2(\al,0,x) & \leq C\big[1 \wedge \frac{\al}{1-\al} \frac{g}{\phi}(x)\big] (1+x^2) \\
 r_2(\al,\mu,x) & \le (1-a(x))\mu^2 + Ca(x)((x-\mu)^2+1).
\end{align*} 
Let $\ga$ be the Cauchy density. For any real $x$ and $\al\in[0,1/2]$,
\begin{align*}
E_{0} r_2(\al,0,x) & \leq C\tau(\al)\al\\
 E_\mu r_2(\al,\mu,x) & \le C(1+\tilde\ta(\al)^2). 
\end{align*}
\end{lem}

The following lower bound is used in the proof of Theorem \ref{negres}. 
\begin{lem} \label{lemlb}
Let $\ga$ be the Laplace density. There exists $C_0>0$ such that, for $x\in \RR$ and $\al\in[0,1]$\[  r_2(\al,0,x) \ge  C_0 \al. \]
\end{lem}

We now turn to bounding $r_2(\hat\al,\mu,x)$. This is the quantity $r_2(\al,\mu,x)$, where $\al$ (which comes in via $a(x)=a_{\al}(x)$) is replaced by $\hat\al$. 
This is done with the help of the threshold $\tilde\ta(\al)$.

\begin{lem}[no signal or small signal]\label{lemns}
Let $\ga$ be the Cauchy density.  Let $\al$ be a fixed non-random element of $(0,1)$. Let $\hat\al$ be a random element of $[0,1]$ that may depend on $x\sim \cN(0,1)$ and on other data. Then there exists $C_1>0$ such that
\[ E r_2(\hat\al,0,x) \le C_1\left[\al\tilde\ta(\al) +  P(\hat\al>\al)^{1/2}\right]. \]
There exists $C_2>0$ such that for any real $\mu$,  if $x\sim \cN(\mu,1)$,
\[ E r_2(\hat\al,\mu,x) \le \mu^2 + C_2.\] 
\end{lem}

\begin{lem}[signal] \label{lemsig}
Let $\ga$ be the Cauchy density. Let $\al$ be a fixed non-random element of $(0,1)$. Let $\hat\al$ be a random element of $[0,1]$ that may depend on $x\sim \cN(\mu,1)$ and on other data and such that $\tilde\tau(\hat\al)^2\le d\log(n)$ with probability $1$ for some $d>0$. Then there exists $C_2>0$ such that for all real $\mu$,
\[ E r_2(\hat\al,\mu,x) \le C_2\left[1 + \tilde\ta(\al)^2 +  (1+ d\log{n})P(\hat\al<\al)^{1/2}\right]. \]
\end{lem}

\subsection{Moments of the score function} \label{sec-mom}

The next three lemmas are borrowed from \cite{js04} and apply to any density $\ga$ such that $\log\ga$ is Lipschitz on $\RR$ and satisfies
\begin{align} 
\ga(y)^{-1}\int_y^\infty\ga(u)du & \approx y^{\kappa-1},\quad\text{as } y\to\infty. \label{tails}
\end{align}
Both Cauchy and Laplace densities satisfy \eqref{tails}, with $\kappa=2$ and $\kappa=1$ respectively, and their logarithm is Lipschitz.

\begin{lem} \label{lemmtilde}
For $\kappa\in[1,2]$ as in  \eqref{tails}, as $\al\to 0$,
\[ \tilde m (\al) \asymp \zeta^{\kappa-1}g(\zeta). \]
Also, the function $\al\to\tilde{m}(\al)$ is nonnegative and increasing in $\al$.
\end{lem}
 
\begin{lem} \label{lemm12}
The function $\al\to m_1(\mu,\al)$ is decreasing in $\al$. Also, $m_1(\zeta,\al) \sim 1/(2\al)$ as $\al\to 0$.  For small enough $\al$, 
\begin{align*}
m_2(\mu,\al) & \leq C \al^{-1}m_1(\mu,\al), \quad \mu\ge 1.
\end{align*}

\end{lem}

\begin{lem} \label{lembeta}
There exist a constant $c_1$ such that for any $x$ and $\al$,
\[ |\beta(x,\al)| \le \frac{1}{\al\wedge c_1}, \]
and constants $c_2,c_3, c_4$ such that for any $\al$, and $\kappa$ as in \eqref{tails},
\begin{align*}
 m_1(\mu,\al) & \le -\tilde{m}(\al) + c_2\zeta(\al)\mu^2, & \quad |\mu|\le 1/\zeta(\al) &\\
 m_1(\mu,\al) & \le (\al \wedge c_3)^{-1} &  \text{ for all } \mu
\end{align*}
and
\begin{align*}
 m_2(\mu,\al) & \le c_4 \frac{\tilde m(\al)}{\zeta(\al)^\kappa \al} &  \quad |\mu|\le 1/\zeta=1/\zeta(\al) \\
 m_2(\mu,\al) & \le (\al \wedge c_3)^{-2} & \quad \text{ for all } \mu.
\end{align*}
\end{lem}


\subsection{In-probability bounds for  $\hat\al$} \label{sec-hal}

Lemma \ref{lemzetaunder} below implies that, for any possible $\te_0$, the estimate $\hat\al$ is smaller than a certain $\al_1$ with high probability.  
One can interpret this as saying that $\hat \al$ does not lead to too much undersmoothing (i.e. too many nonzero  coefficients). On the other hand, if there is enough signal in a certain sense,  $\hat \al$ does not lead to too much oversmoothing (i.e. too many zero coefficients), see Lemma \ref{lemzetaover}. 

Although we generally follow the approach of \cite{js04}, there is one significant difference. One needs a fairly sharp bound on $\al_1$ below. Using the definition from \cite{js04} would lead to a loss in terms of logarithmic factors for the posterior $L^2$--squared moment. So we work with a somewhat different $\al_1$, and shall thus provide a detailed proof of the corresponding Lemma \ref{lemzetaunder}. For the oversmoothing case, one can borrow the corresponding Lemma of \cite{js04} as is.

Let $\al_1=\al_1(d)$ be defined as the solution  of the equation, with $\eta_n=s_n/n$,
\begin{equation} \label{zeta1} 
  d\al_1 \tilde{m}(\al_1) =\eta_n,
\end{equation} 
where $d$ is a constant to be chosen (small enough for Lemma \ref{lemzetaunder} to hold).
A solution of $\eqref{zeta1}$ exists, as using Lemma \ref{lemmtilde}, $\al\to \al \tilde{m}(\al)$ is increasing in $\al$, and equals $0$ at $0$. Also, provided $\eta_n$ is small enough, $\al_1$ can be made smaller than any given arbitrary constant. 
The corresponding threshold $\zeta_1$ is defined by $\be(\zeta_1)=\al_1^{-1}$.
From Lemma \ref{lemmtilde}, we have $\tilde{m}(\al_1)\asymp \zeta g(\zeta_1)$ if  $\ga$ is Cauchy and $\tilde{m}(\al_1)\asymp g(\zeta_1)$ if $\ga$ is Laplace.  

\begin{lem}\label{lemz}
Let $\kappa$ be the constant in \eqref{tails}. Let $\alpha_1$ be defined by \eqref{zeta1} for $d$ a given constant and let $\zeta_1$ be given by $\beta(\zeta_1)=\al_1^{-1}$. Then there exist real constants $c_1, c_2$ such that 
for large enough $n$,
\[ \log(n/s_n) + c_1  
\le \frac{\zeta_1^2}{2} \le \log(n/s_n)  +\frac{\kappa-1}2\log\log{n} +c_2,\]
with $\kappa$ as in \eqref{tails}. Also,  $\zeta_1^2\sim 2\log(n/s_n)$ as $n/s_n$ goes to $\infty$. 
\end{lem}

\begin{lem} \label{lemzetaunder} 
Let $\alpha_1$ be defined by \eqref{zeta1} for $d$ a given small enough constant and let $\zeta_1$ be given by $\beta(\zeta_1)=\al_1^{-1}$. Suppose \eqref{techsn} holds. Then for some constant $C>0$,
\[ \sup_{\te\in\ell_0[s_n] } P_\te[\hat\zeta < \zeta_1] \le \exp(-Cs_n). \]
\end{lem}

For the oversmoothing case, one  denotes
the proportion  of signals above a level $\ta$ by
\begin{equation}\label{pitilde}
 \tilde \pi(\ta;\mu) = \frac1n \#\{i:\, |\mu_i|\ge \ta\}.
\end{equation} 
We also set, recalling that $\al_0$ is defined via $\ta(\al_0)=1$,
\begin{equation} \label{altp}
 \al(\ta,\pi)=\sup\{\al\le \al_0:\ \pi m_1(\ta,\al)\ge 2 \tilde{m}(\al) \}. 
\end{equation} 
One defines $\zeta_{\ta,\pi}$ as the corresponding pseudo-threshold $\be^{-1}(\al(\ta,\pi)^{-1})$.
\begin{lem}[\cite{js04}, Lemma 11] \label{lemzetaover}
There exists $C$ and $\pi_0$ such that if $\pi < \pi_0$, then for all $\ta \ge 1$, 
\[ \sup_{\te:\ \tilde{\pi}(\ta;\te)\ge \pi } P_\te[\hat\zeta > \zeta_{\ta,\pi}] 
\le \exp\{-Cn\phi(\zeta_{\ta,\pi})\}. \]
\end{lem}

\subsection{Proof of Theorem \ref{negres} } \label{prn}

\begin{proof} 
Let $\al^*$ be defined as the solution in $\al$ of the equation, 
\begin{equation}\label{zetaet} 
  \al \tilde{m}(\al) = \eta_n / 4, 
\end{equation} 
where $\eta_n=s_n/n$ (that is $\al^*=\al_1(d)$ with $d=4$ in \eqref{zeta1}).  
Let $\zeta^*$ be defined via $\be(\zeta^*)=\al^*$. 

Let $\te_0$ be the specific signal defined by, for $\al^*, \zeta^*$ as in \eqref{zetaet},
\[ \te_{0,i} = 
\begin{cases}
\ \zeta^*,\qquad & 1\le i\le s_n\\
\ 0, \qquad & s_n<i\le n 
\end{cases}.
\]
Using Lemma \ref{lemmtilde}, one gets $\tilde m(\al^*)\asymp g(\zeta^*) \asymp \gamma(\zeta^*)$ as $\zeta^*\to\infty$. Lemma \ref{lemz} implies ${\zeta^*}^2\ge  2\log(1/\eta_n)+C$, for $C$ a possibly negative constant. Combining this with the definition $\gamma(\zeta^*)=e^{-\zeta^*}/2$ leads to
\begin{equation} \label{baet}
 \al^* \geqa \eta_n e^{\sqrt{\log(1/\eta_n)} },
\end{equation} 
for $c_0$ in \eqref{techsn} small enough to have $2\log(1/\eta_n)+C\ge \log(1/\eta_n)$. We next prove that, for $\hat \al$ given by \eqref{defhal}, for small enough $c>0$,
\begin{equation} \label{tr1}
P_{\te_0}\left[ \hat\al < \al^* \right] \le e^{-cs_n}.
\end{equation}
If $\al^*\le \al_n$ the probability at stake is $0$, as $\hat\al$ belongs to $[\al_n,1]$ by definition. For $\al^*>\al_n$, we have $\{\hat\al < \al^*\}=\{S(\al^*)< 0\}$.
With $A=\sum_{i=1}^n m_1(\mu_i,\al^*)$, 
\[ P_{\te_0}\left[\hat\al<\al^*\right] = P_{\te_0}\left[ S(\al^*)< 0 \right]
= P_{\te_0}\left[ \sum_{i=1}^n \be(\te_{0,i}+Z_i,\al^*) - m_1(\te_{0,i},\al^*) < -A \right] \]
Setting $W_i= m_1(\te_{0,i},\al^*) -  \be(\te_{0,i}+Z_i,\al^*)$, we have $|W_i|\le 2C/\al^*=:M$ and $W_i$ are independent. So by Bernstein's inequality, 
\[ P_{\te_0}\left[ \sum_{i=1}^n W_i > A \right]
 \le \exp\left[ -\frac12 \frac{A^2}{V+\frac13MA} \right],
  \]
 where $V$ is an upper-bound for $\sum_{i=1}^n \text{Var}(W_i)$. The term $A$ equals 
 \[ A = (n-s_n)(-\tilde{m}(\al^*)) + s_n m_1(\zeta^*,\al^*). \]
 The function $\al\to \al \tilde{m}(\al)$ is increasing, as $\tilde m(\cdot)$ is (Lemma \ref{lemmtilde}), so by its definition \eqref{zetaet}, $\al^*$ can be made smaller than any given positive constant, provided $c_0$ in \eqref{techsn} is small enough,  ensuring $\eta_n=s_n/n$ is small enough. Using Lemma \ref{lemm12}, $m_1(\zeta,\al)\sim 1/(2\al)$ as $\al\to 0$. So, using \eqref{zetaet}, one obtains, for small enough $c_0$,
 \[ A \ge \frac{s_n}{3\al^*} -  \frac{s_n}{4\al^*} = \frac{s_n}{12\al^*}. \]
 On the other hand, the last part of Lemma \ref{lembeta} implies
 \begin{align*}
 V & \le \sum_{i\notin S_0} m_2(0,\al^*) + \sum_{i\in S_0} m_2(\zeta^*,\al^*)\\
 & \le C(n-s_n) \frac{\tilde{m}(\al^*)}{\zeta^*\al^*} + C\frac{s_n}{{\al^*}^2}. 
 \end{align*}
 Using the definition of $\al^*$, one deduces $V\leqa s_n/{\al^*}^2$ and from this
 \[ \frac{V}{A^2} + \frac{MA}{3A^2} \leqa \frac{1}{s_n},\]
 which in turn implies \eqref{tr1}, as then
 $P_{\te_0}\left[\sum_{i=1}^n W_i >A\right] \le \exp[-cs_n]$.
 Next one writes
 \begin{align*}
 \int \|\te-\te_0\|^2 d\Pi_{\hat\al}[\te\given X] & 
 \ge \int \|\te-\te_0\|^2 d\Pi_{\hat\al}[\te\given X]  \1_{\hat \al\ge \al^*}\\
 & \ge \sum_{i\notin S_0} \int \te_i^2 d\Pi_{\hat \al}(\te\given X) \1_{\hat \al\ge \al^*}
 \end{align*}
Lemma \ref{lemlb} implies, for any possibly data-dependent weight $\al$, that 
$\int \te_i^2 d\Pi_\al(\te\given X)\geqa \al$, so
\[  \int \|\te-\te_0\|^2 d\Pi_{\hat\al}[\te\given X] \ge (n-s_n)\hat\al \1_{\hat \al\ge \al^*}
\ge (n-s_n)\al^* \1_{\hat \al\ge \al^*}.\]
As $(n-s_n)\al^* P_{\te_0}[\hat\al\ge \al^*]\geqa Cn\al^*(1-e^{-cs_n})$, an application of 
\eqref{baet} concludes the proof.
\end{proof}

\subsection{Proof of Theorem \ref{thm-risk} } \label{prrisk}

Let us decompose the risk $R_n(\te_0)=E_{\te_0} \int \|\te-\te_0\|^2 d\Pi_{\hat\al}(\te\given X)$ according to whether coordinates of $\te$ correspond to a `small' or `large' signal, the threshold being $\zeta_1=\beta^{-1}(\al_1^{-1})$, with $\al_1$ defined in \eqref{zeta1}. One can write
\[ R_n(\te_0)
=\Big[\sum_{i:\ \te_{0,i}=0} + \sum_{i:\ 0<|\te_{0,i}|\le \zeta_1} +
 \sum_{i:\ |\te_{0,i}| > \zeta_1}\Big] E_{\te_0}\int (\te_i-\te_{0,i})^2 d\Pi_{\hal}(\te_i\given X).\]  
We next use the first part of Lemma \ref{lemns} with 
$\al=\al_1$ and the second part of the Lemma to obtain, for any $\te_0$ in $\ell_0[s_n]$,
\begin{align*}
&\lefteqn{\Big[\sum_{i:\ \te_{0,i}=0} + \sum_{i:\ 0<|\te_{0,i}|\le \zeta_1} \Big] E_{\te_0}\int (\te_i-\te_{0,i})^2 d\Pi_{\hal}(\te_i\given X)}\\
& \le C_1 \sum_{i:\ \te_{0,i}=0} \left[\al_1\ta(\al_1) + P_{\te_0}(\hal>\al_1)\right]
+ \sum_{i:\ 0<|\te_{0,i}|\le \zeta_1} (\te_{0,i}^2 + C)\\
& \le C_1\left[(n-s_n)\al_1\ta(\al_1)+(n-s_n)e^{-c_1\log^{2}n} \right] + (\zeta_1^2+C)s_n,
\end{align*}
where for the last inequality we use Lemma \ref{lemzetaunder} and \eqref{techsn}. From \eqref{zeta1} one gets, with $\eta_n=s_n/n$,
\[ n\al_1\leqa n\eta_n  \zeta_1^{-1}g(\zeta_1)^{-1}
\leqa  s_n \zeta_1.\] 
Now using Lemma \ref{lemz} and the fact that $\ta(\al_1)\le \zeta_1$, one obtains that the contribution to the risk of the indices $i$ with $|\te_{0,i}|\le \zeta_1$ is bounded by a constant times $s_n \log(n/s_n)$.

It remains to bound the part of the risk for indexes $i$ with $|\te_{0,i}| > \zeta_1$. To do so, one uses Lemma \ref{lemsig} with $\al$ chosen as $\al=\al_2:=\al(\zeta_1,\pi_1)$ and $\pi_1=\tilde\pi(\zeta_1;\te_0)$, following the definitions \eqref{pitilde}--\eqref{altp}. One denotes by $\zeta_2$ the pseudo-threshold associated to $\al_2$. The following estimates are useful below
\begin{align} 
\zeta_1^2 & < \zeta_2^2     \label{zd} \\
\pi_1 \zeta_2^2 & \le C\eta_n \log(1/\eta_n). \label{zd2}
\end{align}
These  are established in a similar way as in \cite{js04}, but with the updated definition of $\al_1, \zeta_1$ from \eqref{zeta1}, so we include the proof below for completeness. One can now apply Lemma \ref{lemsig} with $\al=\al_2$, 
\begin{align*}
& \lefteqn{\sum_{i:\ |\te_{0,i}|>\zeta_1}
 E_{\te_0}\int (\te_i-\te_{0,i})^2 d\Pi_{\hal}(\te_i\given X)} \\
& \le  C_2 n \pi_1
\left[ 1+ \zeta_2^2 + (1+d\log n)P_{\te_0}(\hal<\al_2)^{1/2}\right]\\
& \le  C_2 n \pi_1
\left[ 1+ \zeta_2^2 + (1+d\log n)P_{\te_0}(\hat\zeta>\zeta_2)^{1/2}\right].
\end{align*}
Let us verify that the term in brackets in the last display is bounded above by $C(1+\zeta_2^2)$. If $\zeta_2>\log{n}$, this is immediate by bounding $P_{\te_0}(\hat\zeta>\zeta_2)$ by $1$. If $\zeta_2\le \log{n}$, Lemma \ref{lemzetaover} implies $P_{\te_0}(\hat\zeta>\zeta_2)\le  \exp(-Cn\phi(\zeta_2))\le \exp(-C\sqrt{n})$, so this is also the case.  
Conclude that the last display is bounded above by $Cn\pi_1(1+\zeta_2^2)\le C'n\pi_1\zeta_2^2$. Using \eqref{zd2}, this term is itself bounded by $Cs_n\log(n/s_n)$, which  concludes the proof of the Theorem, 
given \eqref{zd}--\eqref{zd2}.

We now check that \eqref{zd}--\eqref{zd2} hold.  We first compare $\al_1$ and $\al_2$. For small enough $\al$, the bound on $m_1$ from Lemma \ref{lembeta} becomes $1/\al$, so that, using the definition \eqref{zeta1} of $\al_1$,
\[ \frac{m_1(\zeta_1,\al_1)}{\tilde{m}(\al_1)}
\le \frac1{\al_1}\left(\frac{\eta_n}{d\al_1}\right)^{-1}
\le \frac{d}{\eta_n}\le \frac{d}{\pi_1}, \]
using the rough bound $\pi_1\le \eta_n$. Note that both functions $\tilde{m}(\cdot)^{-1}$ and $m_1(\zeta_1,\cdot)$ are decreasing via Lemmas \ref{lemmtilde}--\ref{lemm12}, and so is their product on the interval where both functions are positive. As $d < 2$, by definition of $\alpha_2$ this means $\al_2< \al_1$ that is $\zeta_1<\zeta_2$.

To prove  \eqref{zd2}, one compares $\zeta_2$ first to a certain $\zeta_3=\zeta(\al_3)$ defined by $\al_3$ (largest) solution of 
\[ \bar\Phi(\zeta(\al_3)-\zeta_1) = \frac8{\pi_1}\al_3\tilde{m}(\al_3), \]
with $\bar\Phi(x)=P[\cN(0,1)>x]$. Using Lemma \ref{lem-lb3}, which also gives the existence of $\zeta_3$, one gets
\[  \frac{m_1(\zeta_1,\al_3)}{\tilde{m}(\al_3)} \ge
\frac{\frac14 \beta(\zeta_3)\bar\Phi(\zeta_3-\zeta_1)}{\tilde{m}(\al_3)}
= \frac1{4\al_3}\frac{8\al_3 \tilde{m}(\al_3)}{\pi_1\tilde{m}(\al_3)} =\frac{2}{\pi_1}.   \]
This shows, reasoning as above, that $\al_3\le \al_2$, that is $\zeta_2\le \zeta_3$. Following \cite{js04}, one distinguishes two cases to further bound $\zeta_3$. 

If $\zeta_3>\zeta_1+1$, using $\zeta_2^2\le \zeta_3^2$ and $\tilde{m}(\al_3)\leqa \zeta_3 g(\zeta_3)$,
\begin{align*}
 \pi_1\zeta_2^2 & \le \zeta_3^2 \frac{8\al_3\tilde{m}(\al_3)}{\bar{\Phi}(\zeta_3-\zeta_1)} \leqa \zeta_3^3 \frac{g(\zeta_3)}{\beta(\zeta_3)} \frac{\zeta_3-\zeta_1}{\phi(\zeta_3-\zeta_1)}\\
 & \le C \zeta_3^4 \frac{\phi(\zeta_3)}{\phi(\zeta_3-\zeta_1)}=C\zeta_3^4\phi(\zeta_1)e^{-(\zeta_3-\zeta_1)\zeta_1}\\
 & \le C (\zeta_1+1)^4e^{-\zeta_1}\phi(\zeta_1),
 \end{align*}
 where for the last inequality we have used that $x\to x^4 e^{-(x-\zeta_1)\zeta_1}$ is decreasing for $x\ge \zeta_1+1$. Lemma \ref{lemz} now implies that 
 $\phi(\zeta_1)\leqa \eta_n$. As $\zeta_1$ goes to $\infty$ with $n/s_n$, one gets $\pi_1\zeta_2^2\leqa \eta_n$.

If $\zeta_1\le \zeta_3\le \zeta_1+1$, let $\zeta_4=\zeta(\al_4)$ with $\al_4$  solution in $\al$ of
\[ \bar\Phi(1) = 8\al\tilde{m}(\al) \pi_1^{-1}.\]
By the definition of $\zeta_3$, since $\bar\Phi(1) \le \bar\Phi(\zeta_3-\zeta_1)$, we have $8\al_4\tilde{m}(\al_4)\le 8\al_3\tilde{m}(\al_3)$ so that $\al_4\le \al_3$. Using Lemma \ref{lemmtilde} as before,   
\[ \bar\Phi(1) \leqa \frac{g(\zeta_4)}{\beta(\zeta_4)}\pi_1^{-1} 
\leqa \phi(\zeta_4)\pi_1^{-1}. \]
Taking logarithms this leads to
\[ \zeta_4^2\le C + 2\log(\pi_1^{-1}).\]
In particular, $\zeta_2^2\le 2\log(\pi_1^{-1})+C$. As $x\to x\log(1/x)$ is increasing, one gets, using  $\pi_1\le \eta_n$,
\[ \pi_1\zeta_2^2 \le 2\eta_n\log(1/\eta_n) + C\eta_n,\]
which concludes the verification of \eqref{zd}--\eqref{zd2} and the proof of Theorem \ref{thm-risk}.

In checking \eqref{zd2}, one needs a lower bound on $m_1$. In \cite{js04}, the authors mention that it follows from their lower bound (82), Lemma 8. But this bound cannot hold uniformly for any smoothing parameter $\al$ (denoted by $w$ in \cite{js04}), as $m_1(\mu,0)=-\tilde m(w)<0$ if $w\neq 0$. So, although the claimed inequality is correct, it does not seem to follow from (82). We state the inequality we use now, and prove it in Section \ref{sec-lemps}.

\begin{lem} \label{lem-lb3}
Let $\bar{\Phi}(t)=\int_t^\infty \phi(u)du$. For $\pi_1, \zeta_1$ as above, a solution $0<\al\le \al_1$ to the equation 
 \begin{equation} \label{wt}
 \bar{\Phi}(\zeta(\al)-\zeta_1) = 8\pi_1^{-1}\al\tilde{m}(\al). 
\end{equation}
exists. Let $\al_3$ be the largest such solution. Then for $c_0$ in \eqref{techsn} small enough, 
\begin{equation} \label{m1}
m_1(\zeta_1,\al_3) \ge \frac14 \beta(\ze_3)\bar\Phi(\al_3-\zeta_1).
\end{equation}
\end{lem}

\subsection{Proof of Theorem \ref{negres2}} \label{prn2}
Let $\te_0, \al^*, \zeta^*$ be defined as in the proof of Theorem \ref{negres}. 
Below we  show that the event $\cA = \{\hat\al\in[\al^*,c\al^*]\}$, for $c$ a large enough constant, has probability going to $1$, faster than a polynomial in $1/n$.  Recall from the proof of Theorem \ref{negres} that, if $\hat\al\ge \al^*$, so in particular on $\cA$, we have $V_X\ge (n-s_n)\al^*\ge n\al^*/2 \ge C_1 s_n g(\zeta^*)^{-1}$. 
Denote
\begin{align*}
v_n & = m s_n g(\zeta^*)^{-1}\\
V_X & = \int \|\te-\te_0\|^2 d\Pi_{\hat\al}(\te\given X),
\end{align*}
where $m$ is chosen small enough so that $v_n\le V_X/2$ on $\cA$. Then,
\begin{align*}
\lefteqn{\Pi_{\hat\al}\left[\|\te-\te_0\|^2 < v_n \given X\right] \1_{\cA} = \Pi_{\hat\al}\left[\|\te-\te_0\|^2 - V_X < v_n -V_X\given X\right] \1_{\cA}}&& \\
& \le \Pi_{\hat\al}\left[\|\te-\te_0\|^2 - V_X < -V_X/2\given X\right] \le 4V_X^{-2} \int \{ \|\te-\te_0\|^2 - V_X \}^2 d\Pi_{\hat\al}(\te \given X), 
\end{align*}
where the second line follows from Markov's inequality.
One now writes the $L^2$--norm in the previous display as sum over coordinates and one expands the square, while noting that given $X$ the posterior $\Pi_{\hat\al}[\cdot\given X]$ makes the coordinates of $\theta$ independent
\begin{align*}
\lefteqn{\int \{ \|\te-\te_0\|^2 - V_X \}^2 d\Pi_{\hat\al}(\te \given X)}&&\\
& = \int \sum_{i,j} \left[(\te_i-\te_{0,i})^2 - \int (\te_i-\te_{0,i})^2 d\Pi_{\hat\al}(\te \given X)\right]\left[(\te_j-\te_{0,j})^2 - \int (\te_j-\te_{0,j})^2 d\Pi_{\hat\al}(\te \given X)\right] d\Pi_{\hat\al}(\te \given X)\\
& =  \sum_{i=1}^n \int \left[(\te_i-\te_{0,i})^2 - \int (\te_i-\te_{0,i})^2 d\Pi_{\hat\al}(\te \given X)\right]^2 d\Pi_{\hat\al}(\te \given X) 
\le \sum_{i=1}^n \int (\te_i-\te_{0,i})^4 d\Pi_{\hat\al}(\te \given X).
\end{align*}
The last bound is the same as in the proof of the upper bound Theorem \ref{thm-risk}, except the fourth moment replaces the second moment. Denote $r_4(\al,\mu,x) = \int (u-\mu)^4 d\pi_\al(u\given x)$, then 
\begin{align*}
r_4(\al,\mu,x) = (1-a(x))\mu^4 + a(x)\int (u-\mu)^4\ga_x(u)du.
\end{align*}
In a similar way as in the proof of Lemma \ref{lemfia}, one obtains $\int (u-\mu)^4\ga_x(u)du\le C(1+(x-\mu)^4)$. Next,  noting that since now $\ga$ is Laplace so $g$ has Laplace tails, $x\to (1+x^4)g(x)$ is integrable, proceeding as in the 
 proof of Lemma \ref{lemfia}, one gets $E_0 r_4(\al,0,x)\leqa \al$ as well as $E_\mu r_4(\al,\mu,x) \leqa 1+ \tilde\ta(\al)^4$, for any fixed $\al$. Similarly as in Lemmas \ref{lemns}--\ref{lemsig}, one then derives the following random $\al$ bounds
 \[ E r_4(\hat\al,0,x) \leqa c\al^* + P(\hat\al>c\al^*)^{1/2}\]
 and, for any $\mu$,
 \[ E r_4(\hat\al,\mu,x) \leqa 1+\ta(\al^*)^4 + (1+\log^2 n)P(\hat\al<\al^*)^{1/2}.\]
By using that the probabilities in the last displays go to $0$ faster than $1/n$, which we show below, and gathering the bounds for all $i$,
\[ E_{\te_0}\sum_{i=1}^n \int (\te_i-\te_{0,i})^4 d\Pi_{\hat\al}(\te \given X)
\leqa s_n(1+\ta(\al^*)^4) + n\al^*. \]
From this deduce that
\begin{align*}
 E_{\te_0} \Pi_{\hat\al}\left[\|\te-\te_0\|^2 < v_n \given X\right] & \leqa P[\cA^c] +  [s_n(1+\ta(\al^*)^4) + n\al^*]/(s_n g(\zeta^*)^{-1})^2 \\
& \leqa P[\cA^c] + s_n^{-1}(1+\ta(\al^*)^4)g(\zeta^*) + s_n^{-1}g(\zeta^*). 
\end{align*} 
The last bound goes to $0$, as $\ta(\al^*)\le \zeta_{\al^*}=\zeta^*$ and $g$ has Laplace tails. To conclude the proof, we show that  $P_{\te_0}(\hat\al\in[\al^*,c\al^*])$ is small.  From the proof of Theorem \ref{negres}, one already has $P_{\te_0}[\hat\al <\al^*]\le \exp(-cs_n)$, which is a $o(1/n)$ using $s_n\geqa \log^2{n}$. To obtain a bound on $P_{\te_0}[\hat\al >c\al^*]$, 
one can now revert the inequalities in the reasoning leading to the Bernstein bound in the proof of Theorem \ref{negres}. With $A=\sum_{i=1}^n m_1(\mu_i,\al)$, we have
\[ P_{\te_0}\left[\hat\al>c\al^*\right] = P_{\te_0}\left[ S(c\al^*)> 0 \right]
= P_{\te_0}\left[ \sum_{i=1}^n \be(\te_{0,i}+Z_i,c\al^*) - m_1(\te_{0,i},c\al^*) > -A \right]. \]
But here, $-A=(n-s_n)\tilde{m}(c\al^*)-s_nm_1(\zeta^*,c\al^*)$. As $\al\to\tilde{m}(\al)$ is increasing, $\tilde{m}(c\al^*)\ge \tilde{m}(\al^*)$. Now by Lemma \ref{lembeta}, 
\[ m_1(\zeta^*,c\al^*) \le (c\al^*\wedge c_3)^{-1}\le  \frac{1}{c\al^*}, \]
provided $\al^*\le c_3/c=c_3/16$, which is the case for $\eta_n$ small enough.  Since by definition $n\tilde{m}(\al^*)=s_n/(4\al^*)$, we have $-A\ge s_n/(8\al^*)$. From there one can carry over the same scheme of proof as for the previous Bernstein inequality, with now $\tilde{A}=-A$ and $\tilde{V}$ the variance proxy 
which is bounded by
\[ \tilde{V} \le (n-s_n) m_2(0,c\al^*) + s_n m_2(\zeta^*,c\al^*)
\leqa n\frac{\tilde{m}(c\al^*)}{\zeta_{c\al^*} c\al^*} + \frac{s_n}{(c\al^*)^2}. \]
Now  $\tilde{m}(c\al^*)\leqa Cg(\zeta_{c\al^*})$. Using bounds similar to those of Lemma \ref{lemz}, one can check that $C_1+\zeta_{\al^*}^2\le \zeta_{c\al^*}^2\le C_2+ \zeta_{\al^*}^2$, which implies that $\tilde{m}(c\al^*)/\zeta_{c\al^*}\leqa  \tilde{m}(\al^*)/\zeta^*\leqa \tilde{m}(\al^*)$. From this one deduces, with $\tilde{M}\le C/s_n$,
\[ \frac{\tilde{V}}{\tilde{A}^2} + \frac{\tilde{M} \tilde{A}}{3\tilde{A}^2} \leqa \frac{C'}{s_n},\]
 which by Bernstein's inequality implies $P_{\te_0}\left[\hat\al>c\al^*\right] \le \exp[-Cs_n]$, which completes the proof of Theorem \ref{negres2}.


\section{Technical lemmas for the SAS prior} \label{sec-tec}

\subsection{Proofs of posterior risk bounds: fixed $\al$}

\begin{proof}[Proof of Lemma \ref{lemfia}]
First one proves the first two bounds. To do so, we
 derive moment bounds on $\ga_x$. Since $\ga_x(\cdot)$ is a density function, we have for any $x$, 
$\int \ga_x(u) du = 1$. This implies 
$ (\log g)'(x)=\int (u-x) \ga_x(u) du = \int u\ga_x(u)du-x$.  
 In \cite{js04}, the authors check, see p. 1623,  that $\int u\ga_x(u)du=:\tilde{m}_1(x)$ is a shrinkage rule, that is $0\le \tilde{m}_1(x)\le x$ for $x\ge 0$, so by symmetry, for any real $x$,
\[ |\int u \ga_x(u) du | \le |x|. \]
Decomposing $u^2 = (u-x)^2+2x(u-x)+x^2$ and noting that 
$\int (u-x)^2 \ga_x(u)du=g''(x)/g(x)+1$, 
\begin{align*}
\int u^2 \ga_x(u) du = \frac{g''}{g}(x) + 1 + 2x \frac{g'}{g}(x)+x^2.
\end{align*}
Note that for $\gamma$ Laplace or Cauchy, we have $|\ga'|\le c_1\ga$ and $|\ga''| \le c_2 \ga$. This leads to 
\[ |g'(x)| = |\int \ga'(x-u)\phi(u)du| \le c_1\int \ga(x-u)\phi(u)du=c_1g(x) \]
and similarly $|g''|\le c_2 g$, 
so that $\int u^2 \ga_x(u) du \le C(1+x^2)$ which gives the first bound using \eqref{sima}. 
We note, {\em en passant}, that the one but last display also implies for any real $x$ that 
\begin{equation} \label{teclb}
\int u^2\ga_x(u)du \ge 1-c_2-2c_1|x|+x^2, 
\end{equation}
which implies that $\int u^2\ga_x(u)du$ goes to $\infty$ with $x$. 
Also, for any real $\mu$, 
\[ \int (u-\mu)^2\ga_x(u)du = (x-\mu)^2 
+  \frac{g''}{g}(x) + 1+ 2(x-\mu) \frac{g'}{g}(x).\]
Now using again $g'/g\le c_1$ and $g''/g\le c_2$ leads to
\[\int (u-\mu)^2\ga_x(u)du  \le C(1+(x-\mu)^2). \]
By using the expression of $r_2(\al,\mu,x)$, this yields the second bound of the lemma.

We now turn to the bounds in expectation. 
 For a zero signal $\mu=0$,  one notes that $x=\ta(\al)$ is the value at which both terms in the minimum in the first inequality of the lemma are equal. So 
\[ E_{0} r_2(\al,0,x) \leqa \int \1_{|x|\le \ta(\al)} \frac{\al}{1-\al} \frac{g}{\phi}(x)\phi(x) (1+x^2)dx
 + \int \1_{|x|> \ta(\al)} (1+x^2)\phi(x)dx.\]
For $\ga$ Cauchy, $g$ has Cauchy tails and $x\to (1+x^2)g(x)$ is bounded, so one gets, with $\al\le 1/2$,
\begin{align*}
E_{0} r_2(\al,0,x) & \leqa \al \int  \1_{|x|\le \ta(\al)} dx + \ta(\al)\phi(\ta(\al))+ 
\phi(\ta(\al))/\ta(\al) \\
& \leqa \ta(\al)\al+ \ta(\al)\phi(\ta(\al)) \leqa \ta(\al)\al+ \ta(\al)\al g(\ta(\al)) \leqa \ta(\al)\al.
\end{align*} 
Turning to the last bound of the lemma, we distinguish two cases. Set for the remaining of the proof  $T:=\tilde\ta(\al)$ for simplicity of notation. The first case is $|\mu|\le 4T$, for which
\[ E_\mu r_2(\al,\mu,x) \le \mu^2 + C \le C_1(1+T^2). \]
The second case is $|\mu|>4T$. 
We bound the expectation of each term in the second bound of the lemma (that for $r_2(\al,\mu,x)$) separately. First, $E[a(x)(1+(x-\mu)^2)]\le C$. It thus suffices to bound $\mu^2E_\mu[1-a(x)]$. To do so, one uses the bound \eqref{postwb} and starts by noting that, if $Z\sim\cN(0,1)$,
\[ E [\1_{|Z+\mu|\le T}] \le P[|Z|\ge |\mu| - T] \le P[|Z|\ge |\mu|/2]. \]
This implies, with $\bar{\Phi}(u)=\int_u^\infty \phi(t)dt\le \phi(u)/u$ for $u>0$, 
\[ E_\mu[ \mu^2 \1_{|x|\le T}] \le C_2|\mu|\phi(|\mu|)\le C_3.\]
If $A=\{ x,\ |x-\mu|\le |\mu|/2 \}$ and $A^c$ denotes its complement,
\[ \sqrt{2\pi}E_\mu[e^{-\frac12 (|x|-T)^2}]\le \int_{A^c} e^{-\frac12(x-\mu)^2} dx
+ \int_A e^{-\frac12 (|x|-T)^2} dx.\]
The first term in the last sum is bounded above by $2\bar{\Phi}(|\mu|/2)$. The second term, as $A\subset\{x,\ |x|\ge |\mu|/2\}$, is bounded above by $2\bar{\Phi}(|\mu|/4)$. This implies, in the case $|\mu|>4T$, that
\[   E_\mu r_2(\al,\mu,x) \le C_4+4\mu^2\bar{\Phi}(|\mu|/4)+5\le C.\]
The last bound of the lemma follows by combining the previous bounds in the two cases. 
\end{proof}

\begin{proof}[Proof of Lemma \ref{lemlb}]
From the expression of $r_2(\al,0,x)$ it follows
\begin{align*}
 r_2(\al,0,x) & \ge a(x) \inf_{x\in\RR} \int u^2 \ga_x(u)du
  \ge \al \frac{g}{\phi \vee g}(x) \inf_{x\in\RR} \int u^2 \ga_x(u)du\\
 & \ge \al \inf_{x\in\RR} \frac{g}{\phi \vee g}(x) \inf_{x\in\RR} \int u^2 \ga_x(u)du  \ge C_0 \al,
\end{align*}
where $c_0>0$. Indeed, both functions whose infimum is taken in the last display are continuous in $x$, are strictly positive for any real $x$, and have respective limits $1$ and $+\infty$ as  $|x| \to \infty$, using \eqref{teclb}, so these functions are bounded below on $\RR$ by positive constants.
\end{proof}

\subsection{Proofs of posterior risk bounds: random $\al$} \label{sec-rab}

\begin{proof}[Proof of Lemma \ref{lemns}]
Using the bound on $r_2(\al,0,x)$ from Lemma \ref{lemfia},
\begin{align*}
r_2(\hat\al,0,x) & = r_2(\hat\al,0,x) \1_{\hat\al\le \al} +  r_2(\hat\al,0,x) \1_{\hat\al >\al} \\
& \le \left[ \frac{\hat \al}{1-\hat \al} \frac{g}{\phi}(x) \wedge 1\right] (1+x^2)\1_{\hat\al\le \al} + C(1+x^2)\1_{\hat\al >\al} \\
& \le \left[ \frac{ \al}{1- \al} \frac{g}{\phi}(x) \wedge 1\right] (1+x^2)\1_{\hat\al\le \al} + C(1+x^2)\1_{\hat\al >\al}.
\end{align*}
For the first term in the last display, one bounds the indicator from above by $1$ and proceeds as in the proof of Lemma \ref{lemfia} to bound its expectation by $C\al\tilde \ta(\al)$. 
The first part of the lemma follows by noting that $E[(1+x^2)\1_{\hat\al >\al}]$ is bounded from above by 
$(2+2E_0[x^4])^{1/2} P(\hat\al>\al)^{1/2}\le C_1 P(\hat\al>\al)^{1/2}$ by Cauchy-Schwarz inequality. The second part of the lemma follows from the fact that using Lemma \ref{lemfia}, 
$r_2(\al,\mu,x)  \le (1-a(x))\mu^2 +C a(x)((x-\mu)^2+1)\le \mu^2 +C(x-\mu)^2+C$ for any $\al$.

\end{proof}

\begin{proof}[Proof of Lemma \ref{lemsig}]
Combining \eqref{postwb} and the third bound of Lemma \ref{lemfia}, 
\[ r_2(\hal,\mu,x)  \le \mu^2 \left[\1_{|x|\le \tilde\ta(\hal)} 
 +   e^{-\frac12 (|x|-\tilde\ta(\hal))^2}\1_{|x| > \tilde\ta(\hal)} \right] + C((x-\mu)^2+1).
\]
Note that it is enough to bound the first term on the right hand side in the last display, as the last one is bounded by a constant under $E_\mu$. Let us distinguish the two cases $\hal\ge \al$ and $\hal<\al$.  

In the case $\hal\ge \al$, as $\tilde\ta(\al)$ is a decreasing function of $\alpha$, 
\begin{align*}
& \lefteqn{\left[\1_{|x|\le \tilde\ta(\hal)} 
 +   e^{-\frac12 (|x|-\tilde\ta(\hal))^2}\1_{|x| > \tilde\ta(\hal)} \right] \1_{\hal\ge \al}}\\
 & \le \left[\1_{|x|\le \tilde\ta(\hal)} + \1_{\tilde\ta(\hal)< |x|\le \tilde\ta(\al)} + 
 e^{-\frac12 (|x|-\tilde\ta(\hal))^2}\1_{|x| > \tilde\ta(\al)} \right] \1_{\hal\ge \al}\\
 & \le \1_{|x|\le \tilde\ta(\al)} + e^{-\frac12 (|x|-\tilde\ta(\al))^2} \1_{|x| > \tilde\ta(\al)},
\end{align*}
where we have used $e^{-\frac12 v^2}\le 1$ for any $v$ and that $e^{-\frac12 (u-c)^2}\le e^{-\frac12 (u-d)^2}$ if $u>d\ge c$. As a consequence, one can borrow the fixed $\alpha$ bound obtained previously so that 
\[ E \left[ r_2(\hal,\mu,x)1_{\hal\ge \al}\right] \le 2 E_\mu r_2(\al,\mu,x) 
\le C\left[1 + \tilde\ta(\al)^2 \right].\]

In the case $\hal < \al$, setting $b_n=\sqrt{d\log n}$ and noting that $\tilde\ta(\hal)\le b_n$ with probability $1$ by assumption, proceeding as above, with $b_n$ now replacing $\tilde\ta(\al)$, one can bound 
\begin{align*}
& \lefteqn{\1_{|x|\le \tilde\ta(\hal)} 
 +   e^{-\frac12 (|x|-\tilde\ta(\hal))^2}\1_{|x| > \tilde\ta(\hal)}}\\
 & \le \1_{|x|\le b_n} + e^{-\frac12 (|x|-b_n)^2} \1_{|x| > b_n}.
\end{align*}
From this one deduces that 
\begin{align*}
&\lefteqn{ E \left( \mu^2\left[\1_{|x|\le \tilde\ta(\hal)} 
 +   e^{-\frac12 (|x|-\tilde\ta(\hal))^2}\1_{|x| > \tilde\ta(\hal)} \right] \1_{\hal<\al}\right) }\\
& \le C \left( E_\mu\left[\mu^4\1_{|x|\le b_n} + \mu^4e^{-(|x|-b_n)^2}\right] \right)^{1/2} P(\hal <\al)^{1/2}.
\end{align*} 
Using similar bounds as in the fixed $\alpha$ case, one obtains
\[ E_\mu\left[\mu^4\1_{|x|\le b_n} + \mu^4e^{-(|x|-b_n)^2} \right] \le C(1+b_n^4). \]
Taking the square root and gathering the different bounds obtained concludes the proof.
\end{proof}

\subsection{Proofs on pseudo-thresholds} \label{sec-lemps}

\begin{proof}[Proof of Lemma \ref{lemz}]
For small $\al$, or equivalently large $\zeta$, we have $(g/\phi)(\zeta)=\be(\zeta)+1 \asymp \beta(\zeta)$. Deduce  that for large $n$, using $\eta_n=d\al_1\tilde{m}(\al_1)$ and
 Lemma \ref{lemmtilde} on $\tilde{m}$,
\[ \eta_n \asymp \al_1 \zeta_1^{\kappa-1} \frac{g(\zeta_1)}{\beta(\zeta_1)} \beta(\zeta_1)
\asymp \zeta_1^{\kappa-1}\phi(\zeta_1) \asymp \zeta_1^{\kappa-1}e^{-\zeta_1^2/2}.\]
From this deduce that
\[ |\log{c} + (\kappa-1)\log{\zeta_1}-\frac{\zeta_1^2}{2}
+\log(1/\eta_n)| \le C.
\]
In particular, using $\log\zeta\le a+\zeta^2/4$ for some constant $a>0$ large enough, one gets $\zeta_1^2\le 4(C+\log(1/\eta_n))\le 4(C+\log{n})$. Inserting this back into the previous inequality leads to
\[ \zeta_1^2/2 \le \log(1/\eta_n)+C + (1/2)(\kappa-1)\log\log{n}. \]
The lower bound is obtained by bounding $(\kappa-1)\log(\zeta_1)\ge 0$, for small enough $\al_1$. 
\end{proof}

\begin{proof}[Proof of Lemma \ref{lemzetaunder}]
Using \eqref{techsn}, $\log(1/\eta_n)\le \log(n)-2\log\log{n}$, and the bound on $\zeta$ from Lemma \ref{lemz}  gives  $\zeta_1^2\le 2\log{n}-\frac{3}{2}\log\log{n}$, so that $t(\al_1)\le \zeta(\al_1)=\zeta_1\le \sqrt{2\log{n}}=t(\al_n)$. It follows that $\al_1$ belongs to the interval $[\al_n,1]$ over which the likelihood is maximised. 

Then one notices  that $\{\hat\zeta<\zeta_1\}=\{\hat \al>\al_1\}=\{S(\al_1)>0\}$, regardless of the fact that the maximiser $\hat\al$ is attained in the interior or at the boundary of $[\al_n,1]$. So
\[ P_\te[\hat\zeta < \zeta_1] = P_\te[S(\al_1)>0].  \]

The score function equals $S(\al)=\sum_{i=1}^n \beta(X_i,\al)$, a sum of independent variables. By Bernstein's inequality, 
if $W_i$ are centered independent variables with $|W_i|\le M$ and $\sum_{i=1}^n\text{Var}(W_i)\le V$, then for any $A>0$,
\[ P\left[ \sum_{i=1}^n W_i >A \right] \le \exp\{-\frac12 A^2/(V+\frac13 MA) \}.\]
Set $W_i=\beta(X_i,\al_1)-m_1(\te_{0,i},\al_1)$ and $A = -\sum_{i=1}^n m_1(\te_{0,i},\al_1)$. Then one can take $M=c_3/\al_1$, using Lemma \ref{lembeta}. 
One can bound $-A$ from above as follows, using the definition of $\al_1$,
\begin{align*}
 -A & \le -\sum_{i\notin S_0} \tilde m(\al_1) + \sum_{i\in S_0} \frac{c}{\al_1} 
  \le -(n-s_n) \tilde m(\al_1) + cs_n/\al_1 \\
 & \le -n \tilde m(\al_1)/2 + cdn\tilde{m}(\al_1)  \le -n \tilde m(\al_1)/4,
\end{align*} 
provided $d$ is chosen small enough and, using again  the definition of $\al_1$,
\begin{align*}
 V & \le \sum_{i\notin S_0} m_2(0,\al_1)  
 + \sum_{i\in S_0} m_2(\te_{0,i},\al_1) 
  \le \frac{C}{\al_1}\left[ (n-s_n) \tilde{m}(\al_1)\zeta_1^{-\kappa} + cs_n/\al_1\right] \\
 & \le C\al_1^{-1}\left[ n\tilde{m}(\al_1)\zeta_1^{-\kappa}/2 + c d n\tilde{m}(\al_1)\right] 
  \le C' d n \tilde{m}(\al_1)/\al_1,
\end{align*} 
where one uses that $\zeta_1^{-1}$ is bounded. 
This leads to 
\[ \frac{V+\frac13 MA}{A^2} \le \frac{C'd}{n \al_1 \tilde m(\al_1)} + \frac{4c_3}{3n \al_1 \tilde m(\al_1)}\le \frac{c_5^{-1}}{n \al_1 \tilde m(\al_1)}.\]
One concludes that $P\left[ \hat \al >\al_1 \right] \le \exp\{ - c_5 n\al_1\tilde{m}(\al_1) \}= \exp\{ - Cs_n \}$ using \eqref{zeta1}.
\end{proof}

\begin{proof}[Proof of Lemma \ref{lem-lb3}]
First we check the existence of a solution. Set $\zeta_\al=\zeta(\al)$ and $R_\al:= \bar\Phi(\zeta_\al-\zeta_1)/(\al \tilde{m}(\al))$. For $\al\to 0$ we have $\zeta_\al-\zeta_1\to\infty$ so by using $\bar\Phi(u)\asymp \phi(u)/u$ as $u\to\infty$ one gets, treating terms depending on $\zeta_1$ as constants and using $\phi(\zeta_\al)\asymp \al g(\zeta_\al)$,
\[ \bar\Phi(\zeta_\al-\zeta_1) \asymp \frac{\phi(\zeta_{\al}-\zeta_1)}{\zeta_\al-\zeta_1} \asymp \al g(\zeta_\al) e^{\zeta_\al\zeta_1}.\] 
As $\tilde{m}(\al)\asymp \zeta_\al g(\zeta_\al)$, one gets $R_\al\asymp e^{\zeta_\al\zeta_1}/\zeta_\al\to \infty$ as $\al\to 0$. On the other hand, with $\pi_1\le s_n/n$  and $\al_1\tilde{m}(\al_1)=ds_n/n$,
\[ R_{\al_1} = \frac{1}{2\al_1\tilde{m}(\al_1)} = \frac{d n}{2s_n} 
\le \frac{8}{\pi_1} \frac{d}{16}, \]
so that $R_{\al_1}< 8/\pi_1$ as $d<2$. This shows that the equation at stake has at least one solution for $\al$ in the interval $(0,\al_1)$. 

By definition of $m_1(\mu,\al)$, for any $\mu$ and $\al$, and $\zeta=\zeta(\al)$,
\begin{align*}
m_1(\mu,\al) & = \int_{-\zeta}^{\zeta} \frac{\beta(x)}{1+\al\be(x)} \phi(x-\mu)dx\ +\ \int_{|x|>\zeta} \frac{\beta(x)}{1+\al\be(x)} \phi(x-\mu)dx\\
 & = \qquad \qquad (A) \qquad \qquad \qquad\ +\ \qquad \qquad \qquad  (B).
 \end{align*}
By definition of $\zeta$, the denominator in (B) is bounded from above by $2\al\be(x)$ so
\[ (B) \ge \frac1{2\al} \int_{|x|>\zeta} \phi(x-\mu)dx \geq \frac12\beta(\zeta)\bar\Phi(\zeta-\mu).\]
One splits the integral (A) in two parts corresponding to $\beta(x)\ge 0$ and $\beta(x)<0$. Let $c$ be the real number such that $g/\phi(c)=1$. By construction the part of the integral (A) with $c\le |x|\le \zeta$ is nonnegative, so, for $\al\le |\be(0)|^{-1}/2$,
\begin{align*}
(A) & \ge \int_{-c}^{c} \frac{\beta(x)}{1+\al\be(x)} \phi(x-\mu)dx\\
& \ge  -\int_{-c}^c \frac{|\be(0)|}{1-\al|\be(0)|} \phi(x-\mu)dx  \\
& \ge -2|\be(0)| \int_{-c}^c \phi(x-\mu)dx,
\end{align*}
where one uses the monotonicity of $y\to y/(1+\al y)$. For $\mu\ge c$, the  integral $\int_{-c}^c \phi(x-\mu)dx$ is bounded above by $2\int_0^c \phi(x-\mu)dx\le 2c\phi(\mu-c)$. To establish \eqref{m1}, it thus suffices to show that
\[ (i):=4|\be(0)| c\phi(\zeta_1-c) \le \frac14\beta(\zeta_3)\bar\Phi(\zeta_3-\zeta_1)=:(ii). \]
The right hand-side equals $2\tilde{m}(\al _3)/\pi_1$ by definition of $\zeta_3$. Since $\ga$ is Cauchy, Lemma \ref{lemmtilde} gives $\tilde{m}(\al _3)\asymp \zeta_3g(\zeta_3)\asymp \zeta_3^{-1}$. It is enough to show that $(\pi_1\zeta_3)^{-1}$ is larger  than $C\phi(\zeta_1-c)$, for suitably large $C>0$. 

Let us distinguish two cases. In the case $\zeta_3\le 2\zeta_1$, the previous claim is obtained, since $\zeta_1$ goes to infinity with $n/s_n$ by Lemma \ref{lemz} and $\phi(\zeta_1-c)=o(\zeta_1^{-1})$. 
In the case $\zeta_3> 2\zeta_1$, we obtain an upper bound on $\zeta_3$ by rewriting the equation defining it. For $t\ge 1$, one has $\bar\Phi(t)\ge C\phi(t)/t$. Since $\zeta_3-\zeta_1> \zeta_1$ in the present case, it follows from the equation defining $\zeta_3$ that
\[ C\frac{\phi(\zeta_3-\zeta_1)}{\zeta_3-\zeta_1} \le 8\al_3\tilde{m}(\al_3)/\pi_1. \]
This can be rewritten using $\phi(\zeta_3-\zeta_1)=\sqrt{2\pi}\phi(\zeta_3)\phi(\zeta_1)e^{\zeta_1\zeta_3}$, as well as $\phi(\zeta_3)=g(\zeta_3)\al_3/(1+\al_3)\geqa \al_3 g(\zeta_3)$ and $\tilde{m}(\al_3)\asymp\zeta_3g(\zeta_3)$. This leads to 
\[\frac{e^{\zeta_1\zeta_3}}{\zeta_3^2}\le \frac{C}{\pi_1}e^{\zeta_1^2/2}.\]
By using $e^x/x^2\ge Ce^{x/2}$ for $x\ge 1$ one obtains $\zeta_1^2e^{\zeta_1\zeta_3/2}\le  e^{\zeta_1^2/2}C/\pi_1$, that is, using $\zeta_1^2\ge 1$,
\[ \pi_1\zeta_3\le \pi_1\zeta_1 + \frac{\pi_1\log(C/\pi_1)}{\zeta_1}\le 
\pi_1\zeta_1 + C\le C'\zeta_1, \] 
using that $u\to u\log(1/u)$ is bounded on $(0,1)$. So the previous claim is also obtained in this case, as $\phi(\zeta_1-c)$ is small compared to $(C'\zeta_1)^{-1}$ for large $\zeta_1$.
\end{proof}

\subsection{Proof of the convergence rate for the modified estimator}

\begin{proof}[Proof of Theorem \ref{thm-risk-mod}]
The proof is overall in the same spirit as that of Theorem 2 in \cite{js04} and goes by 
distinguishing the two cases $s_n\ge \log^2{n}$ and $s_n<\log^2{n}$. The main  difference is that here we work with the full posterior distribution, and the risk bounds require  Lemmas \ref{lemfia}--\ref{lemsig},  that bound the posterior risk in various settings, as well as a result,   Lemma \ref{lem-random-mod} below, in the same vein.

Also, we need to work with a modified version of $\zeta_1$, to make sure that the probability in Lemma \ref{lemzetaunder} goes to 0 fast enough. We note that this version of $\zeta_1$ is the one used in \cite{js04} for both their Theorems 1 and 2 (in our Theorem \ref{thm-risk}, such a modification is not needed and we worked with the simpler version there). 
To do so, one replaces $\eta_n=s_n /n$ in the definition \eqref{zeta1} of $\alpha_1$ by 
\[ \tilde\eta_n=\max\left(\eta_n,\frac{\log^2{n}}{n}\right). \]
To keep notation simple, we still denote the corresponding threshold by $\zeta_1$. 
In the first part of the proof below, $\eta_n\ge \log^2(n)/n$, so this is the same version as in definition \eqref{zeta1}. In the second part of the proof, we have $\tilde\eta_n=\log^2{n}/n$ and we now indicate the relevant properties of the corresponding modified threshold $\zeta_1$. First, the statement of Lemma  \ref{lemz} becomes, with $\kappa=2$ (as $\ga$ is Cauchy),
\begin{equation} \label{modzet}
 \log(1/\tilde{\eta}_n) + c_1 \le \frac{\zeta_1^2}{2} \le
 \log(1/\tilde{\eta}_n) + \frac{1}{2}\log\log{n} + c_2.
\end{equation}
Second, we need below a bound on $P[\hat\zeta<\zeta_1]$ with the modified version of $\zeta_1$ as above. It is not hard to check from the proof of Lemma \ref{lemzetaunder} that this proof goes through with the new version of $\zeta_1$ and $\eta_n$ replaced by $\tilde\eta_n$. The only difference is with the term $cs_n/\al_1$ which is bounded by $cn\tilde{\eta}_n/\al_1=n\tilde{m}(\al_1)$, so that Bernstein's inequality gives 
\begin{equation} \label{dev-zet-mod}
P[\hat\zeta<\zeta_1] \le \exp\{-C'n\al_1\tilde{m}(\al_1)\} \le \exp\{-Cn\tilde{\eta}_n\}
\le e^{-C\log^2{n}}. 
\end{equation}

We are now ready for the proof of Theorem \ref{thm-risk-mod}. First consider the case $s_n\ge \log^2{n}$ and let us show that the risk of the empirical Bayes posterior $\Pi_{\hat\al_A}[\cdot\given X]$ is not larger than that of the non-modified one. One decomposes
\begin{align*}
&\lefteqn{E_{\te_0}\int \| \te -\te_0\|^2 d\Pi_{\hala}(\te\given X)}\\
& = E_{\te_0}\int \| \te -\te_0\|^2 d\Pi_{\hal}(\te\given X) 1_{\hat t\le t_n}
+ E_{\te_0}\int \| \te -\te_0\|^2 d\Pi_{\hala}(\te\given X) 1_{\hat t > t_n} \\
& \le E_{\te_0}\int \| \te -\te_0\|^2 d\Pi_{\hal}(\te\given X)
+ E_{\te_0}\int \| \te -\te_0\|^2 d\Pi_{\al_A}(\te\given X) 1_{\hat t > t_n}=(I)+(II).
\end{align*}
The  term (I) corresponds to the risk of the unmodified estimator, so is bounded as in Theorem \ref{thm-risk}. For  (II), one splits it according to small and large signals $\te_{0,i}$: $(II) = S + \tilde S$, with
\[ S = \sum_{i:\ |\te_{0,i}|\le \zeta_1} E_{\te_0}\int (\te_i-\te_{0,i})^2 d\Pi_{\al_A}(\te_i\given X)1_{\hat t > t_n},\]
and $\tilde S = (II) - S$. From Lemma \ref{lemfia}, one knows that $r_2(\al_A,\mu,x)\le \mu^2 + C(1+(x-\mu)^2)$, while for $\mu=0$, one can use the bound in expectation $E_0r_2(\al,0,x)\le C\al\ta(\al)$, so that 
\[ S \le \{\sum_{i:\ |\te_{0,i}|=0}+\sum_{i:\ 0<|\te_{0,i}|\le \zeta_1} \}E_{\te_0}\int (\te_i-\te_{0,i})^2  d\Pi_{\al_A}(\te_i\given X)
\le Cn\al_A\ta(\al_A) + Cs_n\zeta_1^2. \]
We now use the definition of $\al_A$ to bound $\al_A$ and $\ta(\al_A)$. To bound $\ta(\al_A)$, note that for any $\al\in(0,1)$, by definition $a(\ta(\al))=1/2$, so for a signal of amplitude $\ta(\al)$, the posterior puts $1/2$ of its mass at zero, which means the posterior median is $0$, implying $\ta(\al)\le t(\al)$, so that $\ta(\al_A)\le t_A$. 
Combining with the bound for $\al_A$ of Lemma \ref{lemala}, 
\[ n\al_A\ta(\al_A) \le Cn^{-A}t_A^3.\] 
For any fixed $A>0$, this goes to $0$ with $n$ so it is a $o(s_n\zeta_1^2)$, while $s_n\zeta_1^2$ is bounded by  $Cs_n\log(n/s_n)$ as follows from Lemma \ref{lemz}.
Now to bound $\tilde S$, one adapts the last bound of Lemma  \ref{lemfia} to accommodate for the indicator $1_{\hat t>t_n}$. This is done in Lemma \ref{lem-random-mod} whose bound \eqref{lemriskA} implies 
 $\tilde S\le Cs_n t_A^2
P(\hat t>t_n)^{1/2}$. This bound coincides up to a universal constant with the corresponding bound (128) in \cite{js04} (taken for $p=0$, $\tilde p=1$ and $q=2$,  which corresponds to our setting, i.e. working with $\ell_0$ classes and quadratic risk). So the remaining bounds of \cite{js04} for the case $s_n>c\log^2{n}$ can be used directly (the distinction of the three cases as in \cite{js04} p. 1646-1647 can be reproduced word by word, and is omitted for brevity), leading to $\tilde{S}\le Cs_n\log(n/s_n)$. 

Second, consider the case where $s_n\le \log^2{n}$. We note that for this regime of $s_n$, the inequalities \eqref{modzet} become, using that by definition $\tilde\eta_n=\log^2{n}/n$,
\begin{equation} \label{zet-small-s}
 \log{n}-2\log\log{n} + c_1 
 \le \frac{\zeta_1^2}{2} \le
 \log{n} - \frac{3}{2}\log\log{n} + c_2.
\end{equation}
Let us show that the risk of the plug-in posterior using the modified estimator is at most of the order of the minimax risk. For $\zeta_1$ as above,
\begin{align*}
&\lefteqn{E_{\te_0}\int \| \te - \te_0 \|^2 d\Pi_{\hala}(\te\given X)}\\
& =  
\Big[\sum_{i:\ \te_{0,i}=0} + \sum_{i:\ 0<|\te_{0,i}|\le \zeta_1} +
 \sum_{i:\ |\te_{0,i}| > \zeta_1}\Big] E_{\te_0}\int (\te_i-\te_{0,i})^2 d\Pi_{\hal_A}(\te_i\given X) =: (i) + (ii) + (iii).
\end{align*} 
For the terms (i) and (ii), apply respectively each bound of Lemma \ref{lemns}  with $\al=\al_A$ to get $(ii) \le Cs_n\left[ \zeta_1^2 + 1\right] \le C's_n\zeta_1^2\leqa s_n\log{n}$ using \eqref{modzet}, which is bounded from above by $Cs_n\log(n/s_n)$ in the regime $s_n\le \log^2{n}$. Also,
\[ (i) \le Cn\left[ \al_A\tilde\tau(\al_A) +  P[\hat\al_A >\al_A ]^{1/2}\right]. \]
For large enough $n$, we have $\tilde\tau(\al_A)=\tau(\al_A)$ which is less than $t(\al_A)=t_A$ as noted above. Now $\al_A$ is bounded using Lemma \ref{lemala}, so that $n\al_A\tilde{\ta}(\alpha_A)\leqa t_A(1+A)(\log n)n^{-A}=o(1)$ for $A>0$. 

We now bound the probability $P[ \hat\al_A >\al_A ]^{1/2}$. Recall the inequality $t(\al)^2\ge \zeta(\al)^2-C$ (see e.g. (53) in \cite{js04}). Using \eqref{zet-small-s}, we have $\zeta_1^2\ge 2\log{n}-4\log\log{n}+2c_1$ so, writing in slight abuse of notation $t(\zeta_1)=t(\al_1)$ seeing $t(\cdot)$ as a function of $\zeta_1$ instead of $\al_1$,
\[ t(\zeta_1)^2\ge t_n^2 +\log\log{n}-C+2c_1 \]
so that $t(\zeta_1)\ge t_n$ for $n$ large enough. Deduce $\{\hat \al_A>\al_A\}=\{\hat t < t_n\}\subset\{\hat t<t(\zeta_1)\}=\{\hat\zeta<\zeta_1\}$. Using \eqref{dev-zet-mod}, we have $P[\hat\zeta<\zeta_1]\le e^{-C\log^{2}{n}}$, so that $(i)$ goes to $0$, and so is a $o(s_n\log(n/s_n))$.

Finally, for the term (iii) one uses Lemma \ref{lemsig} with $\al=\al_A$. Note $\{\hat \al_A<\al_A\}=\{t(\hat \al_A)>t_A\}$. But by definition note that $t(\hat \al_A)$ equals either $t_A$ if $\hat t>t_n$ or $t(\hat\al)$ if $\hat t=t(\hat \al)\le t_n$, so that $t(\hat\al)\le t_n$. As $t_n^2<2\log{n}<t_A^2$ for $A>0$, conclude that in all cases $t(\hat\al_A)\le t_A$ with probability one, so that $P[\hat \al_A<\al_A]=0$. Thus
\[ (iii) \le \sum_{i:\ |\te_{0,i}| > \zeta_1} E_{\te_{0,i}} r_2(\hat\al_A,\te_{0,i},X_i)  \le C \sum_{i:\ |\te_{0,i}| > \zeta_1} (1+\tilde\ta(\al_A)^2+0)\le Cs_n\tilde\ta(\al_A)^2,
\]
which is no more than $2Cs_n(1+A)\log{n}\le C's_n\log{n}$. As $s_n\le c\log^2{n}$, we have $\log{n}\leqa \log(n/s_n)$ so 
$(iii)\le Cs_n\log(n/s_n)$. Putting the previous bounds together, one gets $(i)+(ii)+(iii)\le Cs_n\log(n/s_n)$, which concludes the proof. 
\end{proof}

\begin{lem} \label{lemala}
For $A\ge 0$, with $t_A^2=2(1+A)\log{n}$ and $\al_A=t^{-1}(t_A)$,  there exist $N_0>0$ and $C>0$ both independent of $A$ such that for $n\ge N_0$,
\[ \al_A \le C(1+A)(\log{n}) n^{-1-A}.\]
\end{lem}
\begin{proof}
First recall the bound $t(\al)<\zeta(\al)$. Setting $\al=t^{-1}(u)$ in this inequality leads, using $\zeta(u)=\beta^{-1}(1/u)$, to $u<\beta^{-1}(1/t^{-1}(u))$. As $\beta$ is increasing on $\RR^+$, one has $t^{-1}(u)<1/\beta(u)$, so 
\[ \al_A<\frac{1}{\beta(t_A)}=\frac{g}{\phi-g}(t_A)\frac{\phi}{g}(t_A)\le 2\frac{\phi}{g}(t_A)\le Ct_A^2e^{-t_A^2}, \]
where we use that $g$ has Cauchy tails. The result follows by using the expression of $t_A$.
\end{proof}

\begin{lem} \label{lem-random-mod}
For any real $\mu$, for $B:=\{\hat t>t_n\}$, and $\al_A, t_A$ as above,
\begin{equation} \label{lemriskA}   
E_\mu[r_2(\al_A,\mu,x)\1_B] \le C(t_A^2+1)P(B)^{1/2}.
\end{equation}
\end{lem}
\begin{proof}
Similar to the proof of Lemma \ref{lemfia}, one sets $T:=\ta(\al_A)$ and distinguishes two cases: if $|\mu|\le 4T$,  Lemma \ref{lemfia} implies $r_2(\al_A,\mu,x)\le \mu^2+(1+(x-\mu)^2)$, so 
using Cauchy-Schwarz inequality,
\[ E_\mu[r_2(\al_A,\mu,x)\1_B] \le
CT^2 P(B) + P(B) + E_\mu[(x-\mu)^4]^{1/2}P(B)^{1/2}\le C(1+T^2)P(B)^{1/2}.
\]
If $|\mu|>4T$, one uses the bound on $r_2$ from Lemma \ref{lemfia} again keeping the dependence in $a(x)$. First,
\[ E[a(x)\{1+(x-\mu)^2\}\1_B] \le E[\{1+(x-\mu)^2\}^2]^{1/2}P(B)^{1/2}
\le CP(B)^{1/2}. \]
Let us now focus on $E_\mu[(1-a(x))\mu^2\1_B]\le
E_\mu[\{1_{|x|\le T} + e^{-(|x|-T)^2/2}1_{|x| > T}\}\1_B ] $. 
The first term, using $P_\mu[|x|<T]\le \bar\Phi(|\mu|/2)$, is bounded by 
$\mu^2\Phi(|\mu|/2)^{1/2}P(B)^{1/2}\le CP(B)^{1/2}$. The second term is bounded by $\mu^2\{E_\mu[e^{(-|x|-T)^2}]\}^{1/2}P(B)^{1/2}$. In the proof of Lemma \ref{lemfia}, we showed that $E_\mu[e^{(-|x|-T)^2/2}]^{1/2}$ is bounded by a universal constant times $\bar\Phi(|\mu|/4)$. As $e^{-y^2}\le e^{-y^2/2}$, the term at stake is bounded from above by $\mu^2\bar\Phi(|\mu|/4)P(B)^{1/2}\le CP(B)^{1/2}$, which implies \eqref{lemriskA}. 
\end{proof}

\section{Proof of Theorem \ref{thm-lasso}: the SSL prior} \label{sec-sslpr}

Recall that we use the notation of the SAS case, keeping in mind that every instance of $g$ is replaced by $g_1$ and (some of the) $\phi$s by $g_0$. Similarly, $\be(x,\al)$, $\tilde{m}$, $m_1$ and $m_2$ are defined as in Section $\ref{sec-nota}$, but with $\be(x)=g_1/g_0-1$. 

The main steps of the proof generally follow those of Theorem \ref{thm-risk}, although technically there are quite a few differences. In the SSL case, we do not know whether the function $\be=g_1/g_0-1$ is nondecreasing over the whole $\RR^+$. Yet, we managed to show that $\be$, which is an even function, is nondecreasing on the interval
\[ J_n=[2\la_1,\sqrt{2\log{n}}],\] see Proposition \ref{pr1} below. This allows us to define its inverse $\be^{-1}=\be_{|J_n}{}^{-1}$ on this interval. Further, we prove in Lemma \ref{pr2} that $\be$ crosses the horizontal axis on the previous interval,  is strictly negative on $[0,2\la_1]$ and tends to $\infty$ when $x \to \infty$. As $\be$ is continuous, the graph of the function crosses any given horizontal line $y=c$, for any $c>0$.  

{\em The threshold $\zeta$ in the SSL case.}  For every $\al \in (0,1)$, one sets
 \begin{equation}\label{defzet}
 \ze=\ze(\al)=\min\{s>0,\  \be(s)=1/\al\}.
 \end{equation}
This is well defined by the property noted in the previous paragraph. Now one notes that $g_0\le 2\phi$ for $x\le \la_0/2$, see Lemma \ref{2}, and that the function $g_1/\phi$ takes a value at $\sqrt{2\log{n}}$ not smaller than 
$Cn/\log n$, since $g_1\leqa \ga_1$ has Cauchy tails. This implies the existence of a constant $\cC>1$ such that 
\begin{equation}\label{after}
\be(\sqrt{2\log n}) \geq n/(\cC\log n).
\end{equation}  
Now we claim that for any $\al \in (\mathcal{C}\log n/ n, 1]$, we have the identity $\ze(\al)=\be^{-1}(\al^{-1})$. To see this, first note that for any $\al \in (\mathcal{C}\log n/ n, 1]$, by \eqref{after} and $\be(2\la_1)<0$, we have  $\al^{-1} \in \be(J_n)$. This shows that $t=\be^{-1}(\al^{-1})$ solves $\be(t)=\al^{-1}$. Also, it is the smallest possible solution $t>0$, as  $\be$ takes negative values on $[0,2\la_1]$, which establishes the identity. 
 
{\em The threshold $\zeta_1$ in the SSL case.} In the SSL case, the function $\al\to \tilde{m}(\al)=-E_0[\beta(X,\al)]$ is still nondecreasing, since for any real $z$, the map $\cM_z:\al \to z/(1+\al z)$ is nonincreasing and $\beta(X,\al)=\cM_{\be(X)}(\al)$. 
By Proposition \ref{l1}, we also have that $\tilde{m}$ is positive for $\al \geq \mathcal{C}\log n/n$ and is of the order of a constant for $\al=1$. So, the map $\al\to \al\tilde{m}(\al)$ is nondecreasing on $[\mathcal{C}\log n/n,1]$, its value at $\mathcal{C}\log n/n$ is less than $C'\log{n}/n$, and its value at one is of the order of a constant. 
 This shows, using $s_n\ge c_1\log^2{n}$ by \eqref{techsn}, that the following equation has a unique solution $\al_1\in(\cC\log n/n,1)$
\begin{equation} \label{al1ssl}
 \al_1\tilde{m}(\al_1) =ds_n/n,
\end{equation} 
with $d$ a small enough constant to be chosen later (see the proof of Lemma \ref{bernsteinssl}). 
 Thus we can set
\[ \zeta_1 = \beta^{-1}(\al_1^{-1}),\]
and by the above arguments we have $\zeta_1\in J_n$. So Proposition \ref{l1} gives $\al_1^{-1} \asymp \frac{n}{s_n}\ze_1 g_1(\ze_1) \asymp \frac{n}{s_n \ze_1}$.
Now we can follow the same proof as in Lemma \ref{lemz}, replacing up to constants instances of  $g_0(\ze_1)$ by $\phi(\ze_1)$ thanks to Lemma \ref{pr0} and \eqref{dizuit} (as $\ze_1 \le \sqrt{2\log{n}}< \la_0/2$), to obtain
\[ \zeta_1^2 \leqa C\log(n/s_n). \] 

{\em Defining $\ta(\al)$ and $\tilde{\ta}(\al)$}. In the SSL case, we set
\[ \Omega(x,\al)=\frac{\al}{1-\al}\frac{2g_1}{\phi}(x). \]
This definition is as in the SAS case except that $g$ is replaced by $2g_1$. We still use the same notation for simplicity. As $g_1$ satisfies the same properties as $g$, one defines $\ta(\al)$ and $\tilde{\ta}(\al)$ similarly to the SAS case. More precisely, $\ta(\al)$ is the unique solution to the equation $\Omega(\tau(\al),\al)=1$, whenever $\al\le \al^*$, where $\Omega(0,\al^*)=1$. One sets $\ta(\al)=0$ for $\al\ge \al^*$ and $\tilde\ta(\al)=\ta(\al\wedge \al_0)$ with $\ta(\al_0)=\la_1$ (this slightly differs from the SAS case).

As in the proof of Theorem \ref{thm-risk}, one can now decompose the risk $R_n(\te_0)=E_{\te_0} \int \|\te-\te_0\|^2 d\Pi_{\hat\al}(\te\given X)$ according to whether coordinates of $\te$ correspond to a `small' or `large' signal, the threshold being $\zeta_1$ that we define next.
One can write
\[ R_n(\te_0)
=\Big[\sum_{i:\ \te_{0,i}=0} + \sum_{i:\ 0<|\te_{0,i}|\le \zeta_1} +
 \sum_{i:\ |\te_{0,i}| > \zeta_1}\Big] E_{\te_0}\int (\te_i-\te_{0,i})^2 d\Pi_{\hal}(\te_i\given X).\] 
We next use the first part of Lemma \ref{lemsigL} with 
$\al=\al_1$ and the second part of the Lemma to obtain, for any $\te_0$ in $\ell_0[s_n]$,
\begin{align*}
&\lefteqn{\Big[\sum_{i:\ \te_{0,i}=0} + \sum_{i:\ 0<|\te_{0,i}|\le \zeta_1} \Big] E_{\te_0}\int (\te_i-\te_{0,i})^2 d\Pi_{\hal}(\te_i\given X)}\\
& \le C \sum_{i:\ \te_{0,i}=0} \left[\al_1\tilde{\tau}(\al_1) + P_{\te_0}(\hal>\al_1) + \la_0^{-2} \right]
+ \sum_{i:\ 0<|\te_{0,i}|\le \zeta_1} (\te_{0,i}^2 + C)\\
& \le C(n-s_n)\left[\al_1\tilde{\tau}(\al_1)+e^{-C\log^{2}n} +\la_0^{-2}\right] + (\zeta_1^2+C)s_n,
\end{align*}
where for the last inequality we use Lemma \ref{bernsteinssl}. From \eqref{al1ssl} one gets 
\[ n\al_1\leqa s_n \zeta_1^{-1}g(\zeta_1)^{-1}
\leqa  s_n\zeta_1.\] 
Let us now check that $\tilde{\tau}(\al_1)\le \zeta_1$. First, $\be(\zeta_1)=\al_1^{-1}> \al_1^{-1}-1$.  By definition of $\ta(\al_1)$, using $\phi\le 2g_0$ by Lemma \ref{halp}, 
\[ \al_1^{-1}-1=2(g_1/\phi)(\tau(\al_1)) \ge \be(\tau(\al_1))+1. \]
This gives us that $\be(\ze_1)\geq \be(\tau(\al_1))+1$ which implies the result as $\be$ is increasing here. Now with the previous bound on $\zeta_1$  one obtains that the contribution to the risk of the indices $i$ with $|\te_{0,i}|\le \zeta_1$ is bounded by a constant times $s_n \log(n/s_n)$.

It remains to bound the part of the risk for indexes $i$ with $|\te_{0,i}| > \zeta_1$. To do so, one uses the second part of Lemma \ref{lemsigL} with $\al$ chosen as $\al_2'=\cC(\log{n}/n)$, with $\cC$ as in \eqref{after}. 
By definition of $\hat \al$ in \eqref{defhalL}, the probability that $\hat \al$ is smaller than $\al_2'$ equals zero. Also, one has $\tilde\ta(\al_2')^2\leq C\log{n}$. Indeed, setting $\zeta_2'=\beta^{-1}({\al_2'}^{-1})$, we have as before $\ta(\al_2')\le \zeta_2'\le \sqrt{2\log{n}}$. This  implies
\[ \sum_{i:\ |\te_{0,i}|> \zeta_1} E_{\te_0}\int (\te_i-\te_{0,i})^2 d\Pi_{\hal}(\te_i\given X) \le Cs_n \log{n},\]
which concludes the proof of Theorem \ref{thm-lasso}.

\section{Technical lemmas for the SSL prior} \label{seclasso}

\subsection{Fixed $\al$ bounds}

As in the SAS case, we use the notation $\di r_2(\al,\mu,x)= \int (u - \mu)^2d\pi_\al(u\given x)$, where now $\pi_\al(\cdot\given x)$ is the posterior on one coordinate ($X_1$, say) for fixed $\al$ in the SSL case, given $X_1=x$.

\begin{lem} \label{lemfiaL}
For a zero signal $\mu=0$, we have for any $x$ and $\al\in[0,1/2]$,
\begin{align*}  
 r_2(\al,0,x) & \leq C\big[1 \wedge \frac{\al}{1-\al} \frac{g_1}{\phi}(x)\big] (1+x^2)+\int u^2\ga_{0,x}(u)du
  \label{bnosigL} \\
 E_{0} r_2(\al,0,x) & \leq C\tau(\al)\al+{4}/{\la_0^2}.
\end{align*} 
For an arbitrary signal  $\mu\in\RR$, we have that for any real $x$ and $\al\in[0,1/2]$,
\begin{align*}
r_2(\al,\mu,x) & \le (1-a(x))\int (u-\mu)^2\ga_{0,x}(u)du+ Ca(x)((x-\mu)^2+1)\\
 E_\mu r_2(\al,\mu,x) & \le C(1+\tilde\ta(\al)^2). 
\end{align*}
\end{lem}

\begin{proof}
\parindent 0pt \parskip 6pt
By definition, in the SSL case, $\di r_2(\al,0,x) = (1-a(x))\int u^2\ga_{0,x}(u)du+a(x)\int u^2\ga_{1,x}(u)du$.
Similar to Lemma \ref{lemfia}, we have $a(x)\di \int u^2\ga_{1,x}(u)du \leq C\big[1 \wedge \frac{\al}{1-\al} \frac{g_1}{g_0}(x)\big] (1+x^2)$. The first bound now follows from the inequality $g_0\ge \phi/2$ obtained in Lemma \ref{pr0}. For the bound in expectation,
\begin{align*}
 E_{0}\left[\int u^2\ga_{0,x}(u)du\right] & =\int \left(\int u^2\frac{\phi(x-u)\ga_0(u)}{g_0(x)}du\right)\phi(x)dx \\
& \le 2\int u^2\int \phi(x-u)dx\ga_0(u)du=2\int u^2\ga_0(u)du = 4/\la_0^2,
\end{align*}
and one then proceeds as in Lemma \ref{lemfia} to obtain the desired bound for zero signal.
%
%

Now for a general signal $\mu$, the bound for $r_2(\al,\mu,x)$ follows from the definition and the previous bound. For the bound in expectation, {by symmetry one can assume $\mu\ge 0$}. {Also note that the term with the $a(x)$ factor is bounded in expectation by a constant, by using $a(x)\le 1$.} 
To handle the term with $1-a(x)$, we distinguish two cases.
First, one assumes that $\mu\le \la_0/2$. We have, using $(a+b)^2\le 2a^2+2b^2$,
\[ \di (1-a(x))\int (u-\mu)^2\ga_{0,x}(u)du \lesssim (1-a(x))\mu^2 + (1-a(x))\int u^2\phi(x-u)\frac{\ga_0(u)}{g_0(x)}du.\]
For the first term we proceed as in Lemma \ref{lemfia}, for the second {using $g_0\ge \phi/2$ from Lemma \ref{pr0}},  
\begin{align*}
\MoveEqLeft E_{\mu}\left[(1-a(x))\int u^2\phi(x-u)\frac{\ga_0(u)}{g_0(x)}du\right] \le  2 \int u^2 \ga_0(u) \int \frac{\phi(x-u)\phi(x-\mu)}{\phi(x)} dxdu\\
&\lesssim  \int u^2 \ga_0(u) \int e^{-(x-(u+\mu))^2/2+u\mu}dxdu
\lesssim  \la_0 \int u^2 e^{-{\la_0|}u{|}+u\mu}du.
\end{align*}
{As $\mu\le \la_0/2$, } this is in turn bounded by a constant times $(\la_0)^{-2}$. 
Now in the case that $\mu>\la_0/2$, {recall from the proof of Lemma \ref{lemfia}} that for any real $x$, 
\begin{equation} \label{inter1}
 \int (u-\mu)^2\ga_{0,x}(u)du = (x-\mu)^2 
+  1+\frac{g_0''}{g_0}(x) +  2(x-\mu) \frac{g_0'}{g_0}(x).
\end{equation}
The first two terms are, in expectation, bounded by a constant. Next one writes
\[  E_{\mu}\left[(1-a(x))\frac{g_0''}{g_0}(x)\right] = \int (1-a(x))\frac{g_0''}{g_0}(x)\phi(x-\mu)dx \]
{By Lemma \ref{pr0}, we have $|g_0''|=\la_0^2|g_0-\phi|\le 1$. One splits the integral on the last display in two parts. For $|x|\le \mu/4$, one uses that $g_0''$ is bounded together with the bound $g_0\ge \phi/2$. For $|x|> \mu/4$, one uses $g_0''/g_0=\la_0^2(g_0-\phi)/g_0\le \la_0^2$ together with $1-a(x)\le (g_0/g_1)(x)/\al$, which follows from the expression of $a(x)$. This leads to}
\[
E_{\mu}\left[(1-a(x))\frac{g_0''}{g_0}(x)\right]
 \leq \di \int_{|x|\le \mu/4}e^{x\mu-\frac{\mu^2}{2}}dx+\frac{\la_0^2}{\al}\int_{|x|> \mu/4}\frac{g_0}{g_1}(x)\phi(x-\mu)dx.
\]
The first term in the last expression is bounded. The second one is bounded by a constant given our choice of $\la_0$ by combining the following: $\al^{-1}\le n$ from \eqref{defhalL}, $g_0\leqa \ga_0$ for $\mu>\la_0/8$ from  \eqref{dizneuf} and $g_1\geqa \ga_1$.

To conclude the proof, for the last term  in \eqref{inter1}, using \eqref{lax},  the bound on $1-a(x)$ from Lemma \ref{alph} below, and the fact that $x \mapsto x \phi(x)$ is bounded, $E_{\mu}\left[2(1-a(x))(x-\mu) \frac{g_0'}{g_0}(x)\right]$ is bounded by 
\begin{align*}
 \lefteqn{ \di 2\int(1-a(x))|\frac{g_0'}{g_0}(x)||(x-\mu)\phi(x-\mu)|dx\lesssim\di \int(1-a(x))|x|dx}\\
&\lesssim\di \int_{|x|\leq \tilde{\tau}(\al)}|x|dx+\int_{\tilde{\tau}(\al)\leq |x|\leq\frac{\la_0}{2} }|x|e^{-\frac{(|x|-\tilde{\tau}(\al))^2}{2}}dx+\int_{|x|\geq \frac{\la_0}{2}}|x|(1-a(x))dx\\
&\lesssim\di \tilde{\tau}(\al)^2+2(1-e^{-\frac{(\frac{\la_0}{2}-\tilde{\tau}(\al))^2}{2}})+\tilde{\tau}(\al)+\int_{|x|\geq \frac{\la_0}{2}}n^3|x|\frac{\ga_0}{\ga_1}(x)dx \leqa 
1+\tilde{\tau}(\al)^2. \qquad \qedhere
\end{align*}
\end{proof}
 
\begin{lem}\label{alph}
For any $x\in [0,{\la_0}/{2}]$ and $\al\in[0,1]$,
\[ 1-a(x) \leq \1_{|x|\leq\tilde{\tau}(\al)}+4e^{-\frac{1}{2}(|x|-\tilde{\tau}(\al))^2}\1_{|x|>\tilde{\tau}(\al)}. \]
\end{lem}
\begin{proof}
One first notes that $1-a(x) \le 4\Omega(x,\al)^{-1}$ for $x\leq \la_0/2$, using the fact that for such $x$, $g_0(x) \leq 2 \phi(x)$ as found in Lemma \ref{2}. 
The following inequalities hold for $\tilde{\tau}(\al)\le x\le \la_0/2$,  using $\tilde{\ta}(\al)\ge \la_1$ by definition and that $|(\log g_1)'|\le \la_1$ as seen in \eqref{geun}, 
\begin{align*}
 \Omega(x,\al) = & \ \Omega(\tilde{\ta}(\al),\al) \di \exp\left(\int_{\tilde{\tau}(\al)}^x((\log g_1)'(u)-(\log \phi)'(u))du\right)\\
 \ge &\ \exp\left(\int_{\tilde{\tau}(\al)}^x(u-\la_1)du\right)  
\ge \ \exp\left(\int_{\tilde{\tau}(\al)}^x(u-\tilde{\tau}(\al))du\right)
=e^{\frac{(x-\tilde{\tau}(\al))^2}{2}}. \quad \qedhere
\end{align*}
\end{proof}

\subsection{Random $\al$ bounds}

\begin{lem} \label{lemsigL}
Let $\al$ be a fixed non-random element of $(0,1)$. Let $\hat\al$ be a random element of $[0,1]$ that may depend on $x\sim \cN(0,1)$ and on other data. Then there exists $C_1>0$ such that
\[ E r_2(\hat\al,0,x) \le C_1\left[\al\tilde\tau(\al) +  P(\hat\al>\al)^{1/2}\right] 
+ \frac{4}{\la_0^2}. \]
There exists $C_2>0$ such that for any real $\mu$,  if $x\sim \cN(\mu,1)$,
\[ E r_2(\hat\al,\mu,x) \le \mu^2 + C_2.\] 
Suppose now that  $\tilde\tau(\hat\al)^2\le d\log(n)$ with probability $1$ for some $d>0$, and that $x\sim \cN(\mu,1)$. Then there exists $C_2>0$ such that for all real $\mu$,
\[ E r_2(\hat\al,\mu,x) \le C_2\left[1 + \tilde\ta(\al)^2 +  (1+ d\log{n})P(\hat\al<\al)^{1/2}\right]. \]
\end{lem}
\begin{proof}[Proof of Lemma \ref{lemsigL}]
For the first two inequalities, the proof is the same as in the SAS case in Lemma \ref{lemns}, the only difference being the presence of the term $4/\la_0^2$ coming from Lemma \ref{lemfiaL} for the first inequality. 
For the third inequality , it follows from Lemma \ref{lemfiaL} that
 \[ r_2(\hal,\mu,x)  \le \di (1-a_{\hat{\al}}(x)) \int(u-\mu)^2\ga_{0,x}(u)du  + C[(x-\mu)^2+1]. \]
The expectation of the last term is a constant. For the first term, using Lemma \ref{alph},
\[ 1-a_{\hal}(x) \le 
 \1_{|x|\le \tilde\ta(\hal)} 
 +  4 e^{-\frac12 (|x|-\tilde\ta(\hal))^2}\1_{\frac{\la_0}{2}\geq|x| > \tilde\ta(\hal)} + \1_{|x|\geq\frac{\la_0}{2}}n\frac{g_0}{g_1}(x), \]
 where the last estimate uses the bound $\alpha\ge 1/n$. 

As in Lemma \ref{lemsig}, let us distinguish the two cases $\hal\ge \al$ and $\hal<\al$.  
In the case $\hal\ge \al$, as $\tilde\ta(\al)$ is a decreasing function of $\alpha$, 
\begin{align*}
& \lefteqn{\left[\1_{|x|\le \tilde\ta(\hal)} 
 +   4e^{-\frac12 (|x|-\tilde\ta(\hal))^2}\1_{\frac{\la_0}{2}\geq|x| > \tilde\ta(\hal)}  \right] \1_{\hal\ge \al}}\\
 & \lesssim \left[\1_{|x|\le \tilde\ta(\hal)} + \1_{\tilde\ta(\hal)< |x|\le \tilde\ta(\al)} + 
 e^{-\frac12 (|x|-\tilde\ta(\hal))^2}\1_{\frac{\la_0}{2}\geq|x| > \tilde\ta(\al)} \right] \1_{\hal\ge \al}\\
 & \lesssim \1_{|x|\le \tilde\ta(\al)} + e^{-\frac12 (|x|-\tilde\ta(\al))^2}\1_{\frac{\la_0}{2}\geq|x| > \tilde\ta(\al)},
\end{align*}
where we have used $e^{-\frac12 v^2}\le 1$ for any $v$ and that $e^{-\frac12 (u-c)^2}\le e^{-\frac12 (u-d)^2}$ if $u>d\ge c$. 

For the third term, we have to control $E_\mu\left[\di  \1_{|x|\geq\frac{\la_0}{2}}n\frac{g_0}{g_1}(x) \int (u-\mu)^2\ga_{0,x}(u)du\right]$. To do so, one uses \eqref{inter1}.
In expectation, the  term in factor of $(x-\mu)^2+1$ is bounded by a constant. Using \eqref{dizneuf} and the fact that $g_0''/g_0\leq \la_0^2$, the term  in factor  $g_0''/g_0$  is bounded by
\begin{align*}
&\la_0^2 n \di \int_{|x|\geq\frac{\la_0}{2}} \frac{g_0}{g_1}(x) \phi(x-\mu)dx \lesssim n^3\di \int_{|x|\geq\frac{\la_0}{2}} \frac{\ga_0}{\ga_1}(x)dx \\ 
&\lesssim n^4\di \int_{|x|\geq\frac{\la_0}{2}}x^2e^{-\la_0|x|}dx \lesssim n^4 e^{-Cn^2}.
\end{align*} 
Finally, using \eqref{lax} and the fact that $x\mapsto x\phi(x)$ is bounded, one obtains
\[ E_{\mu}\left[\1_{|x|\geq\frac{\la_0}{2}}n\frac{g_0}{g_1}(x)(x-\mu) \frac{g_0'}{g_0}(x)\right]\leq \di \int_{|x|\geq\frac{\la_0}{2}}n\frac{g_0}{g_1}(x)|x| |(x-\mu)\phi(x-\mu)|dx\lesssim \int_{|x|\geq\frac{\la_0}{2}}n\frac{g_0}{g_1}(x)|x| dx.\]
As a consequence, one can borrow the fixed $\alpha$ bound obtained previously so that 
\[ E \left[ r_2(\hal,\mu,x)1_{\hal\ge \al}\right] \lesssim E_\mu r_2(\al,\mu,x) 
\lesssim \left[1 + \tilde\ta(\al)^2 \right].\]
In the case $\hal < \al$, setting $b_n=\sqrt{d\log n}$ and noting that $\tilde\ta(\hal)\le b_n$ with probability $1$ by assumption, proceeding as above, with $b_n$ now replacing $\tilde\ta(\al)$, one can bound 
\begin{align*}
& \lefteqn{\1_{|x|\le \tilde\ta(\hal)} 
 +  4 e^{-\frac12 (|x|-\tilde\ta(\hal))^2}\1_{\frac{\la_0}{2}\geq|x| > \tilde\ta(\hal)} + \1_{|x|\geq\frac{\la_0}{2}}n\frac{g_0}{g_1}(x)}\\
 & \lesssim \1_{|x|\le b_n} + e^{-\frac12 (|x|-b_n)^2} \1_{\frac{\la_0}{2}\geq|x| > b_n} + \1_{|x|\geq\frac{\la_0}{2}}n\frac{g_0}{g_1}(x).
\end{align*}
From this one deduces that $E \left[ \di (1-a_{\hal}(x))\int(u-\mu)^2\ga_{0,x}(u)du \right]$ is bounded from above by a constant times :

$ \left( E_\mu\left[\di \left(\int(u-\mu)^2\ga_{0,x}(u)du\right)^2\1_{|x|\le b_n} + \di \left(\int(u-\mu)^2\ga_{0,x}(u)du\right)^2e^{-(|x|-b_n)^2}\right] \right)^{1/2} P(\hal <\al)^{1/2}$.

Using the same bounds but squared as in the fixed $\alpha$ case, one obtains
\[ E_\mu\left[\di \left(\int(u-\mu)^2\ga_{0,x}(u)du\right)^2\1_{|x|\le b_n} + \di \left(\int(u-\mu)^2\ga_{0,x}(u)du\right)^2e^{-(|x|-b_n)^2} \right] \le C(1+b_n^4). \]
Taking the square root and gathering the different bounds we obtained concludes the proof.
\end{proof}

\subsection{Properties of the functions $g_0$ and $\be$ for the SSL prior}

Recall the notation $\phi, \ga_0,  g_0$ from Section \ref{sec-main}. 
For any real $x$, we also write $\psi(x)=\int_x^{\infty}e^{-u^2/2}du$.
Our key result on $\be$ is the following.
\begin{prop}
\label{pr1}
$\be = \frac{g_1}{g_0}-1$ is {strictly increasing} on $[2\la_1;\sqrt{2\log n}]$.
\end{prop}
We next state and prove some Lemmas used in the proof of Proposition \ref{pr1} below. 
\begin{lem}\label{pr0}
The convolution $g_0=\phi*\ga_0$ satisfies $g_0''=\lambda_0^2(g_0-\phi)\label{gsecond}$ as well as
\[ 
\frac{1}{g_0}\leq \frac{2}{\phi} 
\label{halp}
\qquad \text{and} \qquad 
| g_0-\phi | \leq \frac{1}{\la_0^2}.
\]
\end{lem}
\begin{proof}
The first identity  follows by differentiation. 
One computes $g_0(x)$ by separating the integral in a positive and negative part to get, for any real $x$, 
\begin{equation} \label{gze} 
g_0(x)=\frac{\lambda_0 e^{\frac{\lambda_0^2}{2}}}{2\sqrt{2\pi}}\left[e^{\lambda_0 x}\psi(\lambda_0+x)+e^{-\lambda_0 x}\psi(\lambda_0-x)\right].
\end{equation} 
Now combining the standard inequality 
$(1-x^{-2})e^{-x^2/2}\le  x\psi(x) \leq e^{-x^2/2},$ for $x>0$,
with the expression of $g_0(0)$ obtained from \eqref{gze}, 
we get  $\frac{1}{2}\leq \frac{g_0}{\phi}(0) \leq 1$ for large enough $n$. By \cite{js04}, Lemma 1, the function  $g_0/\phi$ is increasing, which implies the first inequality of the lemma. 

The approximation property of $\phi$ by $g_0$ is obtained by a Taylor expansion. 
For any  $x,u\in\RR$, there exists $c$ between $x$ and $x-u$ such that $\phi(x-u)-\phi(x)=ux\phi(x)+u^2(c^2-1)\phi(c)/2$, so that
$$2(g_0(x)-\phi(x))= \di \int (2ux\phi(x)+u^2(c^2-1)\phi(c))\ga_0(u)du = \di \int u^2(c^2-1)\phi(c)\ga_0(u)du,$$
which is further bounded in absolute value by 
$\di \int u^2|c^2-1|\phi(c)\ga_0(u) du \leq  \int u^2\ga_0(u) du = \la_0^{-2}$.
\end{proof}

\begin{lem} \label{pr3}
Let $L_0=5\sqrt{2\pi}$. Then for all $x \in [0;\sqrt{2\log(\la_0/L_0)}]$, 
\[(\log g_0)'(x) \le -x/2.\]
\end{lem}

\begin{proof}

Let $g_{o+}(x)=\displaystyle \int^{\infty}_0 \phi(v+x) \ga_0(v)dv$ and $g_{o-}(x)=\displaystyle \int_{-\infty}^0 \phi(v+x) \ga_0(v)dv$.
First we check that for any $x$ in the prescribed interval, we have
\[  \la_0(g_{o+}-g_{o-})(x) \le -x(\phi(x)-2/\la_0)\le 0.\]
For any real $x$, {using the inequality $e^v\ge 1+v$}, 
\begin{align*}
g_{o-}(x) = & 
 \int_0^{\infty} \phi(x-u) \ga_0(u)du =  \int_0^{\infty} \phi(x+u)e^{2xu} \ga_0(u)du\\
\geq &\ \  \int_0^{\infty} \phi(x+u)(1+2xu) \ga_0(u)du\\
\geq &\ \  g_{o+}(x)+\la_0 x \int_0^{\infty} u \phi(x+u) e^{-\la_0 u}du.
\end{align*}  
Setting $\Delta(x)=\di \int_0^{\infty} u \phi(x+u) e^{-\la_0 u}du$, one can write
$$\ba 
\Delta(x)&=&\displaystyle \int_0^{\infty} u (\phi(x+u)-\phi(x)) e^{-\la_0 u}du+\phi(x)\int_0^{\infty} u  e^{-\la_0 u}du\\
&=&\displaystyle \int_0^{\infty} u (\phi(x+u)-\phi(x)) e^{-\la_0 u}du+\phi(x)/\la_0^2.\\
\ea$$
As $\phi$ is $1$--Lipshitz,  one can bound from below 
$\phi(x+u)-\phi(x) \ge -u$, which leads to, for any $x\ge 0$,
 \[ \Delta(x)\ge  -\int_0^{\infty} u ^2 e^{-\la_0 u}du+\phi(x)/\la_0^2 \geq-2/\la_0^3+\phi(x)/\la_0^2.\]
 This leads to inequality on $g_{o+}-g_{o-}$ above, using that $x$ belongs to the prescribed interval to get the nonpositivity. 
From this one deduces
\[ g_0'(x) =   \la_0 (g_{o+}-g_{o-})(x)\le -x(\phi(x)-2/\la_0). \]
This now implies 
\[ \frac{g_0'}{g_0}(x) \le -x \frac{\phi(x)-2\la_0^{-1}}{\phi(x) {+}\la_0^{-2}}  \]
On the prescribed interval $\phi(x)\ge 5/\la_0$, so using that $t\to (t-a)/(t+b)$ is increasing, 
\[ \frac{g_0'}{g_0}(x)
\le -x \frac{5\la_0^{-1}-2\la_0^{-1}}{5\la_0^{-1}+\la_0^{-2}}
= -\frac{3x}{5+\la_0^{-1}}\le -\frac{x}{2},
\]   
for large enough $n$, which concludes the proof.
\end{proof} 
%

\begin{proof}[Proof of Proposition \ref{pr1}]

We will firstly note that if $G_1$ has a Cauchy($1/\la_1$) law, 
\begin{equation} \label{geun}
|(\log g_1)'(x)|\leq \la_1.
\end{equation}
Indeed, for any real $x$, recalling that $\ga_1(x)=(\la_1/\pi)(1+\la_1^2x^2)^{-1}$, one sees that $\ga_1'(x)/\ga_1(x)=(-2\la_1^2x)/(1+2\la_1^2x^2)$ and $|\ga_1'(x)/\ga_1(x)|\leq 2\sqrt{2}\la_1/3$. This implies \eqref{geun}, as
\begin{align*}
 |(\log g_1)'(x)| & =|\int \phi(x-u)\ga_1'(u)du|/g_1(x)\\
&\leq \frac{2\sqrt{2}}{3}\la_1 \int \phi(x-u)\ga_1(u)du/g_1(x) \leq \frac{2\sqrt{2}}{3}\la_1 \leq \la_1.
\end{align*}
Let $(x,y) \in [2\la_1;\la_0/4]^2$ with $x \leq y$.
Using {Lemma} \ref{pr3} one can find $c \in [x;y]$ with $\log(g_0(x)/g_0(y))=(x-y) (\log g_0)'(c) \geq (x-y)(-c/2) \geq (y-x)x/2$. On the other hand, 
by \eqref{geun} one deduces that for some $c \in [x;y]$, we have $\log(g_1(x)/g_1(y))=(x-y) (\log g_1)'(c) \leq (y-x)\la_1$. Thus for any $x,y$ as before,
\[\frac{g_1(x)}{g_1(y)} \leq e^{(y-x)\la_1} \qquad\text{and}\qquad e^{(y-x)\frac{x}{2}} \le \frac{g_0(x)}{g_0(y)}.
 \]
As $x \geq 2 \la_1$ by assumption, this leads to the announced inequality.
\end{proof}

\begin{lem}\label{2}
For $n$ large enough, recalling that $\la_0$ depends on $n$, we have 
\begin{align} 
(\log g_0)'(x)  & \ge -x \qquad \text{for any } x>0, \label{lax} \\
g_0(x)  & \le 2 \phi(x)\qquad  \text{for any } 0\le x\le \la_0/2, \label{dizuit}\\
g_0(x) & \lesssim \ga_0(x) \qquad  \text{for any } x \geq {\la_0}/{8}. \label{dizneuf} 
\end{align} 
\end{lem} 
\begin{proof} For any real $x$, we set $\displaystyle \mu_{0,1}(x)=\int u \frac{\phi(x-u)\gamma_0(u)}{g_0(x)}du$, 
 the expectation of $\ga_{0,x}$.
A direct computation shows, for $x>0$ that $(\log g_0)'(x)=-x+\mu_{0,1}(x)$. But
$$\ba 
\mu_{0,1}(x)&=&\displaystyle \int_0^{\infty} u \frac{\la_0 \phi(x-u)e^{-\la_0 u}}{2g_0(x)}du+\int_{-\infty}^0 u \frac{\la_0 \phi(x-u)e^{\la_0 u}}{2g_0(x)}du\\
&=&\displaystyle \int_0^{\infty} u \frac{\la_0 e^{-\la_0 u}}{2g_0(x)}(\phi(x-u)-\phi(x+u))du
=\displaystyle \int_0^{\infty} u \frac{\la_0 e^{-\la_0 u}}{2g_0(x)}\phi(x+u)(e^{2xu}-1)du\geq0
\ea,$$
which leads to \eqref{lax}. 

For the second point, we  first prove the identity, for $x >0$,
 \[ \label{gzero} g_0(x)= \frac{e^{\la_0^2/2}}{\sqrt{2\pi}}\psi(\la_0) \ga_0(x) + \phi(x)\frac{\la_0}{2} \left( e^{{(\la_0-x)^2}/{2}}(\psi(\la_0-x)-\psi(\la_0))+ e^{{(\la_0+x)^2}/{2}}\psi(\la_0+x)\right).\]
Indeed,  $g_0(x)=\displaystyle \int_0^{\infty} \phi(u)(\ga_0(x+u)+\ga_0(x-u))du=\displaystyle \ga_0(x) \int_0^{\infty} \phi(u)e^{-\la_0 u}du+\int_0^{\infty} \phi(u)\ga_0(x-u)du$, for $x>0$. 
The first term equals  ${e^{\la_0^2/2}}\psi(\la_0) \ga_0(x)/{\sqrt{2\pi}}$. The second one equals 
\begin{align*}
\lefteqn{\int_{-x}^{\infty}\phi(x+v)\ga_0(v)dv=\phi(x)\int_{-x}^{\infty}e^{-\frac{v^2}{2}-vx}\ga_0(v)dv} \\
&= \di \phi(x) \frac{\la_0}{2}\left(\int_{-x}^0e^{-\frac{v^2}{2}-vx+\la_0v}dv+\int_{0}^{\infty}e^{-\frac{v^2}{2}-vx-\la_0v}dv\right)\\
&= \di \phi(x)\frac{\la_0}{2}\left(\int_{0}^xe^{-\frac{v^2}{2}+vx-\la_0v}dv+e^{\frac{(x+\la_0)^2}{2}}\int_{0}^{\infty}e^{-\frac{(v+x+\la_0)^2}{2}}dv\right)\\
&= \di \phi(x)\frac{\la_0}{2}\left(e^{\frac{(\la_0-x)^2}{2}}\int_{\la_0-x}^{\la_0}e^{-\frac{u^2}{2}}du+e^{\frac{(x+\la_0)^2}{2}}\psi(x+\la_0)\right)
\end{align*} 
which gives the announced identity. If $x\leq {\la_0}/{2}$, using the inequality 
$y\psi(y) \le e^{-y^2/2}$ for $y>0$, we have
\[ g_0(x)\leq \la_0^{-1}{\ga_0(x)}/{\sqrt{2\pi}}+\phi(x)({\la_0}/{2})\left[({\la_0-x})^{-1}
+{(\la_0+x)}^{-1}\right]. \]
This leads, using $\ga_0(x)/\la_0\le e^{-\la_0^2/2}$ for $x\le \la_0/2$, to $g_0(x)\leq \phi(x)(1/2+1+1/2)=2\phi(x)$.

For the third point, if $x\geq \la_0/8$, the first term is bounded as follows:
$$\ba 
{\la_0 e^{{\la_0^2}/{2}}}e^{\la_0 x} \psi(\la_0+x)&\leq&{\la_0 e^{{\la_0^2}/{2}}}e^{\la_0 x} {e^{-{\la_0^2}/{2}-{x^2}/{2}-\la_0 x}}{(\la_0+x)^{-1}}\\
&\leq&{\la_0 }(\la_0+x)^{-1} e^{-{x^2}/{2}} \le {\la_0 }{(9\la_0/8)^{-1}} e^{-{x^2}/{2}}.
\ea$$

We now bound $\psi(\la_0-x)$ from above by ${e^{-{\la_0^2}/{2}-{x^2}/{2}+\la_0 x}}{(\la_0-x)^{-1}} \leq {4e^{-{\la_0^2}/{2}-{x^2}/{2}+\la_0 x}}{\la_0^{-1}}$ if ${\la_0}/{8}\leq x \leq {3\la_0}/{4}$, which leads to $g_0(x) \lesssim \phi(x)$. 
If $x \geq {3\la_0}/{4}$ one bounds the second term by ${\la_0 e^{{\la_0^2}/{2}-\la_0 x}} \leq {\la_0 e^{{2\la_0 x}/{3}-\la_0 x}}\leq {\la_0 e^{-{\la_0 x}/{3}}}$,  
so that, for $x \geq {\la_0}/{8}$, \[
  g_0(x)\lesssim \ga_0(x). \qquad \qedhere
\]
\end{proof}

The next lemma is useful to control $\be$ outside $[2\la_1,\sqrt{2\log{n}}]$. 
\begin{lem}
\label{pr2}
Set $\la_1=0.05$. For $n$ large enough, for some $C>0$, we have
\begin{align*}
({g_1}/{g_0})(2\la_1) & < 0.25, \\
\be(x)& <0 \qquad & \text{ for all } x\in[0,2\la_1],\\
\be(x)& \gtrsim n/{\log n},  \qquad & \text{ for all } \sqrt{2\log n}\leq x 
\le \la_0/{2}, \\ 
 \be(x) & \gtrsim {e^{Cn^2}}\ga_1(n)/{n}    \qquad & \text{ for all } x \geq \la_0/{8}.
 \end{align*} 
\end{lem}
\begin{proof}
1) We have $\di \frac{g_1}{g_0}(2\la_1)\leq \frac{\la_1\sqrt{2\pi}}{\la_0\int e^{-{(u-2\la_1)^2}/{2}}e^{-\la_0|u|}du}$.
For the denominator, we have 
$$\ba 
\di \int e^{-{(u-2\la_1)^2}/{2}}e^{-\la_0|u|}du&\geq&\di \int_0^{\infty}e^{-{(u-2\la_1)^2}/{2}-\la_0 u}du\\
&\geq&\di e^{{\la_0^2}/{2}-2\la_1\la_0}\int_0^{\infty}e^{-{(u-(2\la_1-\la_0))^2}/{2}}du\\
&\geq&{e^{-2\la_1^2}}\psi(\la_0-2\la_1)/{(\la_0-2\la_1)}\\
&\geq&{e^{-2\la_1^2}}{(\la_0-2\la_1)^{-1}}(1-{(\la_0-2\la_1)^{-2}})\\
&\geq&0.99{e^{-2\la_1^2}}{(\la_0-2\la_1)^{-1}} \text{for }n\text{ large enough}
\ea.$$
This implies $({g_1}/{g_0})(2\la_1)< 0.25$ for $\la_1=0.05$.
~~

2) Let $x \in [0,2\la_1]$, using Lemma \ref{pr0}, we have $\be \leq 2g_1/\phi-1$. As the last function is increasing as we know from the SAS case, we have $\be(x) \leq 2(g_1/\phi)(2\la_1)-1$. With \eqref{dizuit} we end up with $\be(x) \leq 4(g_1/\phi)(2\la_1)-1$, which is strictly negative by the first point.

3) Let $x \in [\sqrt{2\log n},\la_0/2]$. With \eqref{dizuit}, we have $\be(x) \geq (g_1/2\phi)(x)-1 \geq (g_1/2\phi)(\sqrt{2\log n})-1$, and as $g_1 \gtrsim \ga_1$, we end up with $\be(x) \gtrsim n/\log n$.

4) For $x\geq \la_0/8$, via \eqref{dizneuf} we have $\be(x)+1\geq (\ga_1/\ga_0)(x)\geq (\ga_1/\ga_0)(\la_0/8)$ which gives  the result.

\end{proof}

\subsection{Bounds on moments of the score function}\label{momL}

Recall that, for all $k \geq 1$, $\mu \in \R$ and $\alpha \in [0,1]$, $m_k(\mu,\alpha)=E[\be(Z+\mu)^k]$ where $Z \sim \mathcal{N}(0,1)$, 
and $\tilde{m}(\alpha)= - m_1(0,\alpha)= -2 \int_0^{\infty} \be(z,\alpha)\phi(z)dz$.

\begin{prop}\label{l1}
With $\kappa$ as in \eqref{tails}, there exist constants $D_1$ and $D_2$ such that for $\al \in (\mathcal{C}\log n/ n, 1]$, $D_1 \ze^{\kappa -1}g_1(\ze) \leq \tilde{m}(\alpha) \leq D_2 \ze^{\kappa -1}g_1(\ze)$.
Also, $c\leq \tilde{m}(1) \leq C$ with $c,C$ independent of $n$.
\end{prop}
\begin{proof}

Recall that for $\al \in (\mathcal{C}\log n/ n, 1]$, we have $\ze=\be^{-1}(\al^{-1})$ and $\ze \leq \sqrt{2\log n}$.
$$\ba 
\tilde{m}(\alpha)&=& \di -2 \int_0^{\infty} \frac{\be(z)}{1+\alpha \be(z)}\phi(z)dz= \di -2 \int_0^{\infty} \be(z)\phi(z)dz+2 \int_0^{\infty} \frac{\alpha \be^2(z)}{1+\alpha \be(z)}\phi(z)dz\}\\
&=& \di -2 \int_0^{\infty} \be(z)\phi(z)dz+2 \int_0^{\zeta} \frac{\alpha \be^2(z)}{1+\alpha \be(z)}\phi(z)dz+2 \int_{\ze}^{\infty} \frac{\alpha \be^2(z)}{1+\alpha \be(z)}\phi(z)dz:=A+B+C                                                                  
\ea$$

\begin{itemize}
\item For the first term, with $K$ a positive constant one can write : 
$$\ba A&=& \di 2\int_0^{\infty} (\phi-\frac{g_1}{g_0}\phi)= \di 2\int_0^{\infty} (\phi-\frac{g_1}{g_0}(\phi-g_0+g_0))\\
&=& \di 2\int_0^{\infty} (\phi - g_1) + 2\int_0^{\infty} \frac{g_1(g_0-\phi)}{g_0}\\
&=&\di 0+ 2 \int_0^{K\ze}  \frac{g_1(g_0-\phi)}{g_0}+2 \int_{K\ze}^{\infty}  \frac{g_1(g_0-\phi)}{g_0}\\
&:=& (i) + (ii)
\ea$$

Using the fact that $g_1/\phi$ is increasing, we have 

$$\ba
|(i)|&\leq& \di {2}{\la_0^{-2}} \int_0^{K\ze}  {g_1}/{g_0} \leq {4}{\la_0^{-2}} \int_0^{K\ze}  {g_1}/{\phi}\\
&\leq&\di {4K\ze g_1(K\ze)}{\la_0^{-2} /\phi(K\ze)}\lesssim\di {K n^{K^2-2}\ze g_1(K\ze)}\\
\ea$$

Taking $K=6/5$, we end up with $|(i)|\lesssim \ze n^{-2/5}g_1(6\ze /5)$ and this term is strictly dominated by $\ze^{\kappa -1}g_1(\ze)$. By Lemma \ref{pr0}, and the fact that $g_1 \asymp \ga_1$, we have :

$$\ba |(ii)|&\leq& \di 2 \int_{K\ze}^{\infty}g_1(1+{\phi}/{g_0})\leq 6 \int_{K\ze}^{\infty}g_1\\
&\lesssim&(6\ze/5)^{\kappa -1}g_1(6\ze/5)$ using \eqref{tails}$
\ea$$ 
This term too is dominated by $\ze^{\kappa -1}g_1(\ze)$.

~~

\item For the second term, we use the fact that on $(0,\ze)$, $\alpha |\be| <1$, so $1+b_0 \leq 1 + \alpha \be \leq 2$, where $b_0 = {g_1(2\la_1)}/{2\phi(0)}-1$ does not depend on $n$, so that 

$$B \asymp \di \int_0^{\ze} \alpha \be^2(z) \phi(z)dz$$

We will now use the fact that, with $h:={g_1^2}/{\phi}$, $\int_{0}^{\ze}h(z)dz\leq {16}h(\ze)/{\ze}$. This is a direct corollary of lemma 4 in \citep{js04}. We have, also using \eqref{dizuit}: 

\begin{align*}
\lefteqn{\int_0^{\ze} \be^2(z) \phi(z)dz \lesssim \di \int_0^{\ze} ({g_1^2}/{g_0^2})\phi\lesssim \di \int_0^{\ze}{g_1^2}/{\phi}}\\
&&\lesssim \di {g_1^2(\ze)}/{(\ze \phi(\ze))} \lesssim {\be(\ze)g_1(\ze)}/{\ze}\lesssim {g_1(\ze)}{(\alpha \ze)^{-1}}  
\end{align*}

hence $B\lesssim {g_1(\ze)}{\ze^{-1}}$, dominated by $\ze^{\kappa -1}g_1(\ze)$.

~~

\item For the last term, we first use the fact that $\alpha \be(z) < 1 + \alpha \be(z)$ , so that : 
$C \lesssim \di \int_{\ze}^{\infty} \be(z)\phi(z)dz$.

$$\ba 
C &\lesssim& \di \int_{\ze}^{\infty} {g_1}\phi/{g_0} \lesssim \di \int_{\ze}^{\infty}g_1(z)dz$ using Lemma \ref{pr0}$\\
&\asymp& \ze^{\kappa -1}g_1(\ze) $ using \eqref{tails}$
\ea$$

For an upper bound we write $$C=\di 2 \int_{\ze}^{\la_0/2} \frac{\alpha \be^2(z)}{1+\alpha \be(z)}\phi(z)dz+2 \int_{\la_0/2}^{\infty} \frac{\alpha \be^2(z)}{1+\alpha \be(z)}\phi(z)dz=:(i)+(ii).$$

For the first term, using \eqref{dizuit}, we have for every $z \in [\ze,\la_0/2]$, $\be(z)\geq \frac{g_1}{2\phi}(z)-1 \geq \frac{g_1}{4\phi}(z)$ and $\al\frac{g_1}{4\phi}(z)\gtrsim \al \frac{n}{\log n} \gtrsim 1$, so that
$$\ba
(i)&\geq& \di 2 \int_{\ze}^{\la_0/2} \frac{\alpha (g_1^2/16\phi^2)(z)}{1+\alpha (g_1/4\phi)(z)}\phi(z)dz\\
&\gtrsim & \di \int_{\ze}^{\la_0/2} g_1(z)dz \gtrsim \ze^{\kappa-1}g_1(\ze)
\ea$$
For the second term, we have 
$$\ba 
(ii)&\lesssim& \di \int_{\la_0/2}^{\infty} \be(z)\phi(z)dz\\
&\lesssim&\di \int_{\la_0/2}^{\infty}g_1(z)dz \lesssim \la_0^{\kappa-1}g_1(\la_0)\lesssim \la_0^{-1}.
\ea$$
Putting the bounds together finally leads to $\tilde{m}(\alpha)\asymp g_1(\ze)\ze^{\kappa -1}$.

To prove that $\tilde{m}(1) \leq \phi(0)/g_1(2\la_1)$, write $\tilde{m}(1)=\di -2\int_0^{+\infty}\phi+2\int_0^{+\infty}\phi/(1+\be)$.

Now $\di \int_0^{+\infty}\phi/(1+\be)=\int_0^{2\la_1}\phi/(1+\be)+\int_{2\la_1}^{\la_0/2}\phi/(1+\be)+\int_{\la_0/2}^{+\infty}\phi/(1+\be)$.

Using that on $[0,2\la_1]$, $1+\be\geq 1+b_0=g_1(2\la_1)/2\phi(0)$ and \eqref{dizuit} and \eqref{dizneuf}, we have $$\ba \di \int_0^{+\infty}\phi/(1+\be)&\leq& \di \int_0^{2\la_1}\phi/(1+b_0)+\int_{2\la_1}^{\la_0/2}\phi^2/g_1+\int_{\la_0/2}^{+\infty}\ga_0\phi/g_1 \\
&\leq&\di \int_0^{2\la_1}\phi/(1+b_0)+\int_{2\la_1}^{+\infty}\phi^2/g_1+\int_{0}^{+\infty}\phi/g_1\leq C.  \ea $$ 

For the lower bound, recall that $\tilde{m}(1)=\di -2\int_0^{+\infty}\phi+2\int_0^{+\infty}\phi/(1+\be)$ and use Lemma \ref{pr0} to write $\di 2\int_0^{\infty}\phi/(1+\be)\geq \int_0^{+\infty}\phi^2/g_1$ which does not depend on $n$.
\end{itemize}
\end{proof}

\begin{prop}\label{l2} 
Let $\al \in [\mathcal{C}\log n/n,1]$.

1) For small enough $\alpha$, we have $m_2(0,\alpha) \lesssim {\tilde{m}(\alpha)}{(\alpha \ze^{\kappa})^{-1}}$

2) For $k=1$ or $2$, for all $\mu$ and all $\al$ small enough, $m_k(\mu,\al)\leq (\al \wedge |B_0|/(1+B_0))^{-k}$ with $B_0 = {g_1(0)}/{2\phi(0)}-1$.

\end{prop}

\begin{proof}
1) Let $\alpha \in [0;1]$, we have 

$$ \ba 
m_2(0,\alpha) &=& \di 2\int_0^{\infty} \frac{\be^2(z)}{(1+\alpha \be(z))^ 2}\phi(z) dz \\
&=& \di 2\int _0^{\ze} \frac{\be^2(z)}{(1+\alpha \be(z))^ 2}\phi(z) dz + 2\int_{\ze}^{\infty} \frac{\be^2(z)}{(1+\alpha \be(z))^ 2}\phi(z) dz\\

\ea$$

For the first term, as in Proposition \ref{l1}, and using Proposition \ref{pr0}, we have $$\di \int _0^{\ze} \frac{\be^2(z)}{(1+\alpha \be(z))^ 2}\phi(z) dz \lesssim  \int_0^{\ze} \be^2(z) \phi(z) dz \lesssim {g_1(\ze)}{(\alpha \ze)^{-1}}$$

For the last term, by the fact that $\be$ is increasing on $ [\ze,\sqrt{2\log n}]$, \eqref{dizuit} and \eqref{dizneuf} we have that $\be>0$ on $[\ze,\infty]$ so that $$\di \int_{\ze}^{\infty} \frac{\be^2(z)}{(1+\alpha \be(z))^ 2}\phi(z) dz \lesssim 1/\al^2\int_{\ze}^{\infty} \phi(z) dz \lesssim \be^2(\ze){\phi(\ze)}/{\ze}\lesssim \be(\ze){g_1(\ze)}/{\ze}$$

hence $m_2(0,\alpha) \lesssim \frac{g_1(\ze)}{ \alpha\ze}$.
Yet $\tilde{m}(\alpha) \asymp \ze^{\kappa-1}g_1(\ze)$ when $\alpha \to 0$, which yields the first point.

2) Recall the definition $m_k(\mu,\al)= \int \left(\frac{\be(t)}{1+\al\be(t)}\right)^k \phi(t-\mu)dt$. 
If $\be(t)\geq 0$, $\left|\frac{\be(t)}{1+\al\be(t)}\right| \leq 1/\al$. Otherwise we have $|t|< \la_0/2$ so using \eqref{dizuit} for the numerator leads to $\be(t) \geq  {g_1(0)}/{2\phi(0)}-1=B_0 $ and for the denominator $|1+\al\be(t)|=1+\al\be(t) \geq 1+\be(t) \geq 1+B_0$.
\end{proof}

\subsection{In-probability bounds}

\begin{lem} \label{bernsteinssl}
We take $\al=\al_1$ and $\ze=\ze_1$ as defined by \eqref{al1ssl}.
There exists  $C>0$ such that 
\[ \sup_{\te \in \ell_0(s_n)} P_{\te}(\hat{\ze}<\ze) \leq \exp(-C(\log n)^2).\]
\end{lem}

\begin{proof}
First note that, almost surely, $\hat{\al}^{-1}\geq 1>\be(2\la_1)$ with the help of the first point of Lemma \ref{pr2}, so $\hat{\ze}=\be^{-1}(\hat{\al}^{-1})>2\la_1$. Since $\be$ is increasing on $(2\la_1,\sqrt{2\log n})$ and $\ze \leq \sqrt{2\log n}$, we have $\{\hat{\ze}<\ze\}=\{\hat{\alpha}>\alpha\}$, so $P(\hat{\ze}<\ze)=P(\hat{\alpha}>\alpha)=P(\hat{\alpha}>\alpha \cap S(\alpha)>0)+P(\hat{\alpha}>\alpha \cap S(\alpha)\leq 0)$.

Let us now focus on the event $\{\hat{\alpha}>\alpha \} \cap \{S(\alpha)\leq 0\}$. If $S(\alpha) \leq 0$, since $S$ is decreasing, $S < 0$ on $]\alpha, \hat{\alpha}]$. So the likelihood $l$ is decreasing on $]\alpha, \hat{\alpha}[$. It implies that there exists $\alpha' \in ]\alpha,\hat{\alpha}[$ such that $l(\alpha')>l(\hat{\alpha})$. But this contradicts the maximality of $\hat{\alpha}$. Therefore $\{\hat{\alpha}>\alpha \} \cap \{S(\alpha)\leq 0\}= \emptyset$.
Hence $P(\hat{\ze}<\ze)=P(\hat{\alpha}>\alpha \cap S(\alpha)>0) \leq P(S(\alpha) > 0)$.

The score function $S(\alpha)=\sum_{i=1}^n \be(\te_i + Z_i, \alpha)$ is a sum of independent random variables, each bounded by $\alpha^{-1}$. We have $P(S(\alpha)>0)=P(\sum_{i=1}^n W_i > A)$, with $A= - \sum_{i=1}^n m_1(\te_i, \alpha)$ and $W_i = \be(\te_i + Z_i, \alpha) - m_1(\te_i,\alpha)$ centered variables, bounded by $M=(1+c)/\alpha$ using the second point of Proposition \ref{l2}. Setting $V= \sum_{i=1}^n var(W_i)$, Bernstein's inequality gives 
$$\di P(S(\alpha)>0) \leq \exp(\frac{-A^2}{2(V+\frac{MA}{3})}).$$
Moreover, proceeding as in Lemma \ref{lemzetaunder} in the SAS case, we have $-A \lesssim -n \tilde{m}(\alpha)$ and $V \lesssim n\frac{\tilde{m}(\alpha)}{\alpha }$, so $\left(\frac{A^2}{2(V+\frac{MA}{3})}\right)^{-1}=\frac{V}{A^2}+\frac{M}{3A}\leq \frac{C}{\alpha n \tilde{m}(\alpha)}+\frac{C'}{\alpha n \tilde{m}(\alpha)} \lesssim (\alpha n \tilde{m}(\alpha))^{-1}$ therefore $\frac{A^2}{2(V+\frac{MA}{3})}\gtrsim \alpha n \tilde{m}(\alpha) \gtrsim {s_n}\gtrsim (\log n)^2$ and finally 
$$\di P(S(\alpha)>0) \leq \exp(-C (\log n)^2).$$
\end{proof}

\bibliographystyle{abbrv}

\bibliography{bibebb}

\end{document}